\documentclass[11pt,reqno]{amsart}

\usepackage{fullpage}
                \usepackage[margin=2.75cm]{geometry}

\usepackage{amssymb,amsfonts}
\usepackage[all,arc]{xy}
\usepackage{enumerate}
\usepackage{mathrsfs}
\usepackage{color}   
\usepackage{hyperref}
\hypersetup{
    colorlinks=true, 
    linktoc=all,     
    linkcolor=blue,  
}
\usepackage{amsthm}
\usepackage{amsmath}
\usepackage{graphicx}
\usepackage{wasysym}
\usepackage{pgf,tikz}
\usetikzlibrary{positioning}

\usepackage{amsthm}
\usepackage{amsmath}
\usepackage{graphicx}
\usepackage{wasysym}
\usepackage{pgf,tikz}
\usetikzlibrary{positioning}
\usepackage{epsfig,bm,epsf,float}
\usepackage{calc}
\usepackage{mathtools}

\usepackage{xypic}
\usepackage[toc,page]{appendix}
\usepackage{amscd}
\usepackage{tikz-cd}
\usetikzlibrary{arrows,positioning}

\usepackage[english]{babel}
\usepackage{caption}
\usepackage{subcaption}

\newcommand{\cc}{\mathbb C}

\newcommand{\zz}{\mathbb Z}

\newcommand{\qq}{\mathbb Q}
\newcommand{\rr}{\mathbb R}

\newcommand{\la}{\langle}
\newcommand{\ra}{\rangle}
\newcommand{\lra}{\longrightarrow}
\newcommand{\hra}{\hookrightarrow}

\newcommand{\al}{\alpha}
\newcommand{\be}{\beta}
\newcommand{\ga}{\gamma}
\newcommand{\de}{\delta}
\newcommand{\Del}{\Delta}
\newcommand{\ep}{\epsilon}
\newcommand{\vp}{\varpi}
\newcommand{\lam}{\lambda}
\newcommand{\Lam}{\Lambda}
\newcommand{\hd}{\mathrm{hd}}

\makeatletter
\def\Ddots{\mathinner{\mkern1mu\raise\p@
\vbox{\kern7\p@\hbox{.}}\mkern2mu
\raise4\p@\hbox{.}\mkern2mu\raise7\p@\hbox{.}\mkern1mu}}
\makeatother

\newcommand{\B}{\mathcal{B}}

\renewcommand{\bm}{{\bf m}}
\newcommand{\bn}{{\bf n}}
\newcommand{\zi}{z_{\i}(b_\bullet)}
\renewcommand{\i}{\underline{i}}

\DeclareMathOperator{\ind}{ind}
\DeclareMathOperator{\G}{G}

\DeclareMathOperator{\Ind}{Ind}

\DeclareMathOperator{\SL}{SL}

\DeclareMathOperator{\Lie}{Lie}

\DeclareMathOperator{\val}{val}

\newcommand{\fg}{\mathfrak g}

\newcommand{\fh}{\mathfrak h}

\newcommand{\fs}{\mathfrak{s}}
\newcommand{\fu}{\mathfrak{u}}

\newcommand{\calo}{\mathcal{O}}

\newcommand{\Gn}{\overline{\mathrm{G}}^{(n)}}

\newtheorem{Thm}{Theorem}[section]
\newtheorem{Prop}[Thm]{Proposition}
\newtheorem{Algo}[Thm]{Algorithm}
\newtheorem{Lem}[Thm]{Lemma}
\newtheorem{Cor}[Thm]{Corollary}
\newtheorem{Conj}[Thm]{Conjecture}

\theoremstyle{definition}
\newtheorem{Def}[Thm]{Definition}

\theoremstyle{remark}
\newtheorem{Rem}[Thm]{Remark}

\theoremstyle{definition}

\title{Resonant Mirkovi\'{c}-Vilonen polytopes and formulas for highest-weight characters}

\author{Spencer Leslie}

\date\today

\address{Department of Mathematics, Duke University, Durham, NC 27710, USA}

\email{lesliew@math.duke.edu}


\begin{document}

\begin{abstract}
Formulas for the product of an irreducible character $\chi_\lam$ of a complex Lie group and a deformation of the Weyl denominator as a sum over the crystal $\B(\lam+\rho)$ go back to Tokuyama. We study the geometry underlying such formulas using the expansion of spherical Whittaker functions of $p$-adic groups as a sum over the canonical basis $\B(-\infty)$, which we show may be understood as arising from tropicalization of certain toric charts that appear in the theory of total positivity and cluster algebras. We use this to express the terms of the expansion in terms of the corresponding Mirkovi\'{c}-Vilonen polytope.

In this non-archimedean setting, we identify resonance as the appropriate analogue of total positivity, and introduce \emph{resonant Mirkovi\'{c}-Vilonen polytopes} as the corresponding geometric context. Focusing on the exceptional group $G_2$, we show that these polytopes carry new crystal graph structures which we use to compute a new Tokuyama-type formula as a sum over $\B(\lam+\rho)$ plus a geometric error term coming from finitely many crystals of resonant polytopes.

\end{abstract}

\maketitle

\tableofcontents
\section{Introduction}\label{Section: introduction}
\subsection{Overview}
Let $\chi_\lam$ be the character of an irreducible reprensentation $V(\lam)$ of highest-weight $\lam$ of a complex Lie group $G^\vee$. Tautologically, one may write $\chi_\lam$ as a sum over the Kashiwara crystal graph $\B(\lam)$: for $\tau$ in a maximal torus $T^\vee\subset G^\vee$ 
\[
\chi_\lam(\tau)=\sum_{\nu\in\B(\lam)}\tau^{\mathrm{wt}(\nu)},
\]
where $wt:\B(\lam)\lra X^\ast(T^\vee)$ is the weight map of the crystal graph. Remarkably, one may also write $\chi_\lam$ as a sum over $\B(\lam+\rho)$, where $\rho$ is the Weyl vector. Indeed, several such formulas are known going back to Tokuyama, who gave such an expression in type $A$ by studying certain statistics of strict Gelfand-Tsetlin patterns with top row $\lam+\rho$. In each case, one must introduce a \emph{deformation of the Weyl denominator} to modify $\chi_\lam$. By the Casselman-Shalika formula, values of spherical Whittaker functions $W_\tau$ of the $p$-adic group $G$ dual to $G^\vee$ can be expressed in terms of such characters, and in fact the same deformation appears naturally in this context. It is thus natural to seek a direct connection between Whittaker functions and crystal graphs. 

The goal of this paper is to develop and analyze, using ideas from geometric representation theory, the connection between geometry and the values of Whittaker functions. We begin by considering the formal expansion of these functions as a sum of integrals indexed by Mirkovi\'{c}-Vilonen (MV) cycles first introduced in \cite{McN}, which is related to a construction of Braverman, Finkelberg and Gaitsgory \cite{BravGaits, BFG}. Using Lusztig's parametrization, MV cycles may be put in bijection with the canoncial basis $\B(-\infty)$ of the upper part of the quantized enveloping algebra (see Section \ref{Sec: canonical} below), so that $W_\tau$ is a formal sum over an infinite crystal graph. Indeed, McNamara reproved Tokuyama's result in type $A$ by reducing to a sum over $\B(\lam+\rho)$, and conjectured that this was possible for a general group $G$. We will show that this is false in general, and that the failure of this to hold reveals new structures on MV polytopes (see \cite{anderson2003polytope}).

We show that this expansion may be understood as arising from the tropicalization of certain toric charts that appear in the theory of total positivity and cluster algebras. These toric charts are geometric liftings of Lusztig's parametrization of $\B(-\infty)$, and by introducing them we take advantage of the theory of generalized minors to explicitly compute the terms arising in the expansion in terms of the corresponding MV polytope. Our computation reveals that, in general, not only does the sum over $\B(-\infty)$ not reduce to $\B(\lam+\rho)$, but is not even a finite sum: there are infinitely many MV polytopes with non-trival contributions. Moreover, we also find that the sum \emph{cannot} be reduced via cancellation to $\B(\lam+\rho)$.

We explain this behavior in terms of \emph{resonance}. This notion has combinatorial origins in the literature on metaplectic Whittaker functions, but we reinterpret it geometrically from the point of view of total positivity: given a positive structure on a variety $X$, one tropicalizes $X$ with respect to a toric chart $\mathbb{G}_m^N\hra X$ by restricting to ``positive loops'' $\mathbb{G}_m^N(\rr_{>0}((t)))$ and taking valuations. The benefit of this restriction is that rational maps then tropicalize by the na{i}ve rules
\[
\val(x+y) = \min\{\val(x),\val(y)\}, \quad \val(xy)=\val(x)+\val(y).
\]
In the $p$-adic setting, this first rule fails on sets of positive measure when $\val(x)=\val(y)$, and this leads to resonance when studying $p$-adic integrals. Our perspective is that the behavior of these integrals may be lifted to additional structure on the relevant MV polytopes. To make this precise, we introduce the notion of a resonant MV polytope to give a geometric framework.

To illustrate these ideas, we consider Whittaker functions on the exceptional group $G_2$. In this case, we construct new crystal graph structures on the set of (infinitely many) resonant MV polytopes that rely on the resonant properties. We then use this structure to reduce the sum over $\B(-\infty)$ to a sum over $\B(\lam+\rho)$ plus a geometric error term coming from finitely many crystals of resonant polytopes, proving a new geometric Tokuyama-type formula. This gives the first example of new results on canonical bases  proved by considering the connection with $p$-adic Whittaker functions.

Next, we present our main results precisely.

\subsection{Whittaker functions, MV cycles, and Tokuyama-type formulas}  Let $G$ be a simply-connected Chevalley group over a non-archimedean local field $F$ with ring of integers $\calo$ and uniformizer $\vp$. Fix a split maximal torus $T\subset G$ and a Borel subgroup $B\supset T$. These choices determine the set of roots $\Phi=\Phi(G,T)$ and positive roots $\Phi^+$, with simple roots $\Delta$. One also has the Langlands dual groups $T^\vee\subset B^\vee\subset G^\vee$, where $G^\vee$ is the simple complex Lie group with dual root system and $T^\vee$ is a maximal torus dual to $T$. For each $\tau\in T^\vee$, one obtains a character of $B$ which is trivial on $B\cap K$, where $K\subset G$ is a maximal compact subgroup of $G$. We define the associated spherical function $\varphi_\tau:G\lra \cc$ by
\[
\varphi_\tau(bk)=\left(\delta^{1/2}\tau\right)(b),
\]
where $b\in B$ and $k\in K$, and $\de$ is the modular quasi-character of $B$. This is sufficient to define $\varphi_\tau$ by the Iwasawa decomposition $G=BK$. If $\psi:U\lra \cc^\times$ is a generic character, we are interested in the spherical Whittaker function of $G$
\begin{equation}\label{eqn: whittaker intro}
W_\tau(g) := \displaystyle \int_{U}\varphi_\tau(w_0ug)\psi(u)du,
\end{equation}
where $w_0$ is the long element of the Weyl group of $G$. These objects are ubiquitous in several areas of mathematics, and have been a central subject of research for several decades. For example, the explicit evaluation of $W_\tau$ is a crucial step in the Langlands-Shahidi method for studying automorphic $L$-functions. Additionally, if we replace the group $G$ with a metaplectic cover $\widetilde{G}$, these are precisely the $p$-parts of certain Weyl group multiple Dirichlet series \cite{BBF1,ChintaGunnells, McNCS}.  We recall the theory of metaplectic Whittaker functions in Section \ref{Section: Whittaker}.%

The Iwasawa decomposition implies that $W_\tau$ is completely determined by its values on the image of $\vp$ under the dominant coweight $\lam:\mathbb{G}_m\lra T$, which we denote by $\vp^\lam$. By duality, $\lam$ corresponds to the irreducible representation $V_\lam$ of $G^\vee$ with highest weight $\lam$. The famous formula of Casselman and Shalika \cite{CassShalika} connects these two, stating that
\begin{equation}\label{Casselman-Shalika}
W_\tau(\vp^\lam)=\prod_{\al\in \Phi^+}\left(1-q^{-1}\tau^\al\right)\chi_\lam(\tau),
\end{equation}
where $q=|\calo/\vp|$, and $\chi_\lam$ is the character of the irreducible representation $V_\lam$. 

As stated above, Tokuyama independently studied the product on the right-hand side, giving a combinatorial computation in type $A$ for it as a sum over {strict Gelfand-Tsetlin patterns} with top row $\lam+\rho$ \cite{Tok}. This may be interpreted \cite{BBF1} as a sum over the Kashiwara crystal graph $\B(\lam+\rho)$ of highest weight $\lam+\rho$:
\begin{equation}\label{Tokuyama's formula}
\prod_{\al\in \Phi^+}\left(1-q^{-1}z^\al\right)\chi_\lam(z)= \sum_{\nu\in \B(\lam+\rho)}G_{std}(\nu)z^{wt(\nu)-w_0(\lam+\rho)},
\end{equation}
where $G_{std}$, known as the standard contribution function, is defined in terms of the crystal graph structure of $\B(\lam+\rho)$. This formula has ties to several areas in representation theory \cite{GuptaRoyvanPaski,BBFybe} and may be viewed as a discrete analogue of the connection between archimedean Whittaker functions and {geometric crystals} \cite{chhaibi2013littelmann,Lam}.

Finding generalizations of this formula for other groups is a difficult problem in combinatorial representation theory that has seen much recent interest \cite{HamelKing,friedlander2015crystal,defranco}. Viewed as identities in the polynomial ring $\cc[q^{-1}][T]$ where $q^{-1}$ is a formal parameter, each case expresses the product (\ref{Casselman-Shalika}) as a sum over $\B(\lam+\rho)$, though with increasingly more complex replacements for the contribution function $G_{std}$. 

In \cite{McN}, McNamara related Whittaker functions to crystals more directly by connecting the computation of $W_\tau(\vp^\lam)$ to the geometric construction of crystal bases via \emph{Mirkovi\'{c}-Vilonen (MV) cycles}. More precisely, for any Chevalley group $G$, let $\i=(i_1,\cdots,i_N)$ is a reduced expression of the long element $w_0$ of the Weyl group. Then there is a decomposition (depending on $\i$) of the unipotent radical of the Borel 
\begin{equation}\label{eqn: decomp intro}
U=\bigsqcup_{{\bf m}\in \zz_{\geq0}^N}C^{\i}(\bf m),
\end{equation}
where $\bm$ is defined by taking valuations of certain coordinates $\{t_\al,w_\al\}_{\al\in \Phi^+}$. The key point is that the same decomposition arises in the study of MV cycles and polytopes \cite{Kam1}: the data $\bm$ corresponds to the \textbf{$\i$-Lusztig data} of the MV polytope. With this the Whittaker function (\ref{eqn: whittaker intro}) may be expressed formally as a sum of integrals
\begin{equation}\label{eqn: whittaker sum}
W_\tau(\vp^\lam)=\sum_{\bm\in \B(-\infty)}I_\lam(\bm)
\end{equation}
indexed by (stable) MV polytopes of $G$, or more representation-theoretically the canonical basis  $\B(-\infty)$, where
\begin{equation}\label{intro mv int}
I_\lam(\bm)=\int_{C^{\i}(\bm)}\varphi_\tau(w_0u\vp^\lam)\psi(u)du.
\end{equation}

The reason this is formal is that while the Casselman-Shalika formula tells us that $W_\tau(\vp^\lam)$ is a polynomial, it is not at all clear that the sum on the right is finite. Indeed, essentially nothing is known about the support of this sum or the evaluations of the individual terms in general. However, in the special case $G=\SL(r)$ and for the lexicographically minimal reduced expression of $w_0$, McNamara shows that
\begin{equation}\label{eqn: vanishing}
I_\lam(\bm)=0 \text{  unless  }(\i,\bm)\in \B(\lam+\rho),
\end{equation}
and recovers Tokuyama's formula term by term.  Thus, the connection between spherical Whittaker functions and the theory of MV polytopes recovers Tokuyama's formula as a special case, so it is thus natural to study this expansion for other classes of algebraic groups and long words.
\begin{Rem}
Similar results have been achieved in special cases using the theory of Fourier coefficients of Eisenstein series, including a new Tokuyama-type formula in type $B$ in \cite{ FZ16}. By considering metaplectic Eisenstein series, crystal-graph theoretic expressions for the $p$-parts of Weyl group multiple Dirichlet series for types $A$ \cite{BBF1, BBF} and $C$ \cite{FZ15,gray2017metaplectic} have also been established. However, in all these cases the approach is ultimately combinatorial and the connection of this work to MV cycles is not clear.
\end{Rem}
\subsection{Main results}
Fix a dominant coweight $\lam=\sum_i\lam_i\Lam_i^\vee$, where $\Lam_i^\vee$ is the dual basis to the simple roots, and set $-w_0\lam=\sum_i\lam^\ast_i\Lam^\vee_i$.  While the expansion (\ref{eqn: whittaker sum}) directly relates spherical Whittaker function to $\B(-\infty)$, the determination of those basis elements for which $I_\lam(\bm)\neq0$ and more generally the explicit computation of $I_\lam(\bm)$ are highly non-trivial. McNamara conjectured \cite[Remark 8.5]{McN} that the behavior (\ref{eqn: vanishing}) in the type $A$ case is general: that is, for any dominant weight $\lam$, 
\begin{equation}\label{conj mcnamara}
I_\lam(\bm)=0 \text{  unless  }(\i,\bm)\in \B(\lam+\rho).
\end{equation}
for any Chevalley group $G$ and reduced expression $\i$. This would constitute a vast generalization of Tokuyama's formula even in type $A$, where this is known only for a single long word. 

We will show that this is false in general: there can be infinitely many canonical basis elements $(\i,\bm)\notin \B(\lam+\rho)$ such that $I_\lam(\bm)\neq0$. Moreover, these ``external terms'' need not cancel out, so that it is not even true that the sum (\ref{eqn: whittaker sum}) reduces to a sum over $\B(\lam+\rho)$. From our perspective, this explains the increasingly complicated nature of the contribution functions $G_{std}$: the geometric expansion of (\ref{Casselman-Shalika}) in terms of crystal graphs is not naturally a sum over $\B(\lam+\rho)$ but over some larger set. 

As noted above, the central obstruction here is \emph{resonance}, which we introduce \emph{resonant MV polytopes} to give a geometric context (see Figure \ref{fig:test2intro} for an example for $G_2$). Fix an MV polytope $M$ of weight $(0,\mu)$ and $\i$-Lusztig data $\bm$, and consider the associated integral $I_\lam(\bm)$.  Expressing this integral in terms of $M$ is already a difficult problem: while the term $\varphi_\tau(w_0u\vp^\lam)$ appearing in (\ref{intro mv int}) is easy to compute, expanding $\psi(u)$ in terms of $(\i,\bm)$ is highly intricate and requires new tools. Write
$$\psi(u)=\psi\left(\mathfrak{s}_1(u)+\cdots+\mathfrak{s}_r(u)\right),$$ where $\mathfrak{s}_i:U\lra  U_{\al_i}\cong F$ is the natural projection onto the $i^{th}$ simple root group. Using the connection to MV cycles, McNamara showed \cite[Thm 7.8]{McN} that $(\i,\bm)\in B(\lam+\rho)$ if and only if $\val(\fs_i( u))\geq-\lam_i^{\ast}-1$ for all $u\in C^{\i}(\bm)$. If these valuations are constant on $C^{\i}(\bm)$, the integral is more well-behaved. Conceptually, the MV polytope $M$ is \emph{resonant} with respect to $\i$ if these valuations vary on positive-measure subsets of $C^{\i}(\bm)$, though we need to express this in terms of $M$. 

To do this, we must compute $\fs_i(u)$ in terms of the natural coordinates $\{t_\al,w_\al\}_{\al\in \Phi^+}$ defining the ``cell'' $C^{\i}(\bm)\subset U$. In fact, we will show that for ``most'' polytopes $M$, these coordinates are geometrically encoded by certain embeddings of an algebraic torus
\[
\mathbb{G}_m^N\hookrightarrow U;
\] this is Theorem \ref{Thm: geom algo}. With this, we apply the theory of generalized minors as developed by Berenstein, Fomin, and Zelevinsky \cite{BZ97, BZ01, BFZ} to express the functions $\mathfrak{s}_i(u)$ as Laurent polynomials with {positive} integer coefficients and monomials indexed by certain paths in the weight polytope of the corresponding fundamental representation $V({\Lam_i)}$. These paths are $\i$-trails, a representation-theoretic tool introduced by Berenstein and Zelevinsky \cite{BZ01}; see Section \ref{Section: Prelim} for their definition. 
  \begin{Thm}\label{Thm: intro character sum}
  Let $u\in C^{\i}(\bm)$ such that $m_k>0$ for all $k$. For each $i$, the function $\mathfrak{s}_i(u)$ is a Laurent polynomial in $\{w_\al\}$ with positive integer coefficients.
This Laurent polynomial is a sum over $\underline{i}$-trails $\pi: \Lam_i\lra w_0s_i\Lam_i,$
\[
\mathfrak{s}_i(u)=\sum_{\pi: \Lam_i\lra w_0s_i\Lam_i}d_\pi {w_1^{k_1}\cdots w_N^{k_N}},
\]
where $k_i\in \zz$ is determined by $\pi$, $d_\pi\in \zz_{>0}$ (see Prop \ref{Prop: character sum} for more details).
  \end{Thm}

\begin{Rem}\label{Remark intro}
These toric charts were studied \cite{Lusztigpositivity,BZ97, BZ01, BFZ} in the context of total positivity and led to the notion of a {positive structure}, which has recently played an important role in the theory of cluster algebras and Teichm\"{u}ller theory (see \cite{GoncharovShen}, for example). A positive structure on a variety (or stack) $X$ comes equipped with a functorial notion of {tropicalization}, given by restricting to the Zariski-dense set of \emph{transcendental (or generic)} points of the space of formal loops in $X$ and taking valuations. It is precisely the inability to do this with the nonarchimedean analogue that gives rise to the phenomenon of resonance.
\end{Rem}
 As the set up of \cite{McN} includes covering groups, the results of Section \ref{Section: Application} apply in this generality. As a corollary of Theorem \ref{Thm: intro character sum}, we derive a simple characterization of the when an $\i$-Lusztig datum lies in the highest weight crystal $\B(\lam)$. 

 To extend these results to all Lusztig data, we will augment the algorithm McNamara gives to construct the decomposition (\ref{eqn: decomp intro}) in Section \ref{Section: algorithm} so that for each $i$ there is a Laurent polynomial $h_i(t_\al,w_\al)\in \zz\left[t_\al,w_\al^{\pm1}:\al\in \Phi^+\right]$ such that $\mathfrak{s}_i(u)=h_i(t_\al,w_\al)$. The various cases of the algorithm then correspond to different specializations of the variables; for example, the case of positive Lusztig data corresponds to the diagonal specialization $h_i(t_\al,t_\al)$. In Section \ref{Section: boundary}, we show how to recover a Laurent polynomial $g_i(t_\al,w_\al)$ from this diagonal specialization $h_i(t_\al,t_\al)$ which is in some sense as good an approximation to $h_i$ as the diagonal specialization can give. In particular, Proposition \ref{Prop: some words} tells us that $g_i(t_\al,w_\al)=h_i(t_\al,w_\al)$ for certain words. One interesting such case when is the exceptional group $G_2$ with the long word $\i=(2,1,2,1,2,1)$, where $\al_1$ is the long simple root and $\al_2$ is the short simple root; even in this low-rank case, working directly with the algorithm for the Iwasawa decomposition is prohibitive. We illustrate these computations in Section \ref{Section: G2 integrals} by computing the appropriate $\i$-trails in this case and write down the $2$ resulting Laurent polynomials. We conjecture that the polynomials $g_i(t_\al,w_\al)$ are sufficient to compute $I_\lam(\bm)$; see Conjecture \ref{Conjecture}. In Appendix \ref{App: type A}, we verify this in the unique $\i$ for $\SL(4)$ for which $g_i\neq h_i$.

Theorem \ref{Thm: intro character sum} gives a geometric way of understanding for resonance:  we say $M$ is \textbf{resonant} with respect to $\i$ if there are two $\i$-trails $\pi$ and $\pi'$ such that the valuations of the associated monomials $X_\pi,\; X_{\pi'}$ agree. If this is the case, there is a positive-measure subset of $C^{\i}(\bm)$ such that the terms ``resonate'':
\[
\val(X_\pi+X_{\pi'})>\val(X_\pi)+\val(X_{\pi'}).
\]
 Such a resonance corresponds to additional symmetries among the $\i$-Lusztig data of $M$. For the case of $G_2$ and $\i=(2,1,2,1,2,1)$, we find that the resonance of interest is when $m_3=m_5$, where $\bm=(m_1,\ldots,m_6)$ is the $\i$-Lusztig data.
 \begin{figure}

\begin{minipage}{.5\textwidth}
  \centering
  \includegraphics[width=.6\linewidth]{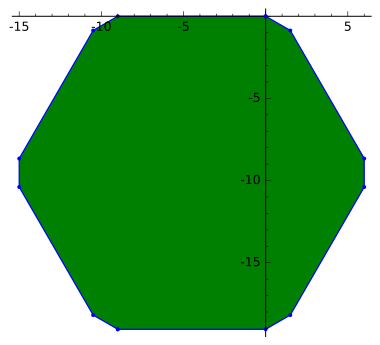}
  \caption{Resonant MV polytope for $G_2$}
  \label{fig:test2intro}
\end{minipage}
\end{figure} See Section \ref{Section: resonance} for a more complete discussion, where we conjecture that the notion of resonance completely encodes the failure of Conjecture \ref{conj mcnamara}.

\begin{Rem}
Resonance originates from the study of Fourier coefficients of metaplectic Eisenstein series (most notably in \cite{BBF1} and \cite{FZ15}). One of our motivations has been to understand the geometry behind this phenomenon. The study of these coefficients is the natural generalization of Langlands-Shahidi method and central to the theory of automorphic forms on covering groups, but requires new methods as these are not Euler products and their local components are mysterious in general \cite{Gaothesis, BBF, BF}. As the Whittaker functions we study are intimately related to these objects, we expect the structures we introduce to shed new light on this important subject.

\end{Rem}

To illustrate the importance of this notion, we consider the case of $G_2$ and $\i=(2,1,2,1,2,1)$ in full detail and, using the results of Section \ref{Section: Application}, compute the integrals $I_\lam(\bm)$. Our motivation is that $G_2$ as it is both complex enough to exhibit many new phenomena as well as small enough to compute in full. As we are primarily interested understanding the geometry of resonant MV polytopes and how it relates to Whittaker functions, we state our results in the non-covering case; the covering group case may also be considered with similar (albeit more complicated) results. We summarize this computation in Theorem \ref{Thm: all together}, which states that inside $\B(\lam+\rho)$, $I_\lam(\bm)$ has a uniform expression in terms of the combinatorics of the crystal graph, whereas $I_\lam(\bm)=0$ outside $\B(\lam+\rho)$ unless $\bm$ has the resonance mentioned above. The upshot is that infinitely many resonant MV polytopes with $\i$-Lusztig data $(\i,\bm)\notin \B(\lam+\rho)$ have non-vanishing contribution.

The resulting power series expansion of the spherical Whittaker function $W_\tau(\vp^\lam)\in\cc[q^{-1}][[T^\vee]]$ must reduce to a polynomial in $\tau\in T^\vee$ by the Casselman-Shalika formula; in particular, there should be systematic cancellation in the power series. Amazingly, we prove this cancellation by constructing a new crystal graph structures on the set of resonant MV polytopes. In Section \ref{Section: resonance families}, we study the family of resonant MV polytopes $RMV(\mu)$ of fixed weight $(0,\mu)$ which contribute to (\ref{eqn: whittaker sum}). We show that there exists additional structure on this set: it may be decomposed into a disjoint union of (truncated) Kashiwara crystal graphs.
\begin{Thm}\label{Thm: crystal intro}
For a fixed weight $\mu$, the set $RMV(\mu)$ of resonant MV-polytopes of weight $\mu$ may be equipped with operations  $$e_1,e_2,f_1,f_2:RMV(\mu)\lra RMV(\mu)\cup \{0\},$$  which decompose the set into a disjoint union of (truncated) Kashiwara crystal graphs of type $A_1\times A_1$. 
\end{Thm}
These crystal graphs look like parallelograms with the height and width given by the highest weight with each vertex representing to a resonant MV polytope and the edges corresponding to Kashiwara operators defined using the additional symmetries of resonant polytopes. For example, if the highest weight is $(2,3)$ the crystal will look like Figure \ref{Figure 2} while a truncated crystal is only contained in such a crystal graph as in Figure \ref{Figure 3}. 
 \begin{figure}
\centering
\begin{minipage}{.4\textwidth}\centering
\begin{tikzpicture}[
  every node/.style = {
    draw, circle, 
    inner sep = 1pt
  },
  thick arrow/.style = {
    ->, -latex, 
    ultra thick
  },
  thin arrow/.style = {
    ->, -stealth', 
  },
  ]

  \node (L0) {};

  \node[below left = of L0] (L1 0){};
  \node[below right = of L0] (L1 1){};

  \node[below left = of L1 0] (L2 0){};
  \node[below right= of L1 0] (L2 1){};
  \node[below right= of L1 1] (L2 2){};

  \node[below right = of L2 0] (L3 1){};
  \node[below right = of L2 1] (L3 2){};
   \node[below right = of L2 2] (L3 3){};
  \node[below right = of L3 1] (L4 0){};
  \node[below right = of L3 2] (L4 1){};
%
  \node[below right = of L4 0] (L5 0){};
%
%

  \draw[thin arrow] (L0) -- (L1 0);
  \draw[thick arrow] (L0) -- (L1 1);

  \draw[thin arrow] (L1 0) -- (L2 0);
  \draw[thick arrow] (L1 0) -- (L2 1);
  \draw[thin arrow] (L1 1) -- (L2 1);
   \draw[thick arrow] (L1 1) -- (L2 2);
%
  \draw[thick arrow] (L2 0) -- (L3 1);
  \draw[thin arrow] (L2 1) -- (L3 1);
  \draw[thick arrow] (L2 1) -- (L3 2);
    \draw[thin arrow] (L2 2) -- (L3 2);
      \draw[thick arrow] (L2 2) -- (L3 3);
%

  \draw[thick arrow] (L3 1) -- (L4 0);
  \draw[thin arrow] (L3 2) -- (L4 0);
  \draw[thick arrow] (L3 2) -- (L4 1);
\draw[thin arrow](L3 3)--(L4 1);
  \draw[thick arrow] (L4 0) -- (L5 0);
  \draw[thin arrow] (L4 1) -- (L5 0);
%
%

\end{tikzpicture} \captionof{figure}{$A_1\times A_1$  crystal graph}\label{Figure 2}\end{minipage}
\begin{minipage}{.4\textwidth}\centering

\begin{tikzpicture}[
  every node/.style = {
    draw, circle, 
    inner sep = 1pt
  },
  thick arrow/.style = {
    ->, -latex, 
    ultra thick
  },
  thin arrow/.style = {
    ->, -stealth', 
  },
  ]

  \node (L0) {};

  \node[below left = of L0] (L1 0){};
  \node[below right = of L0] (L1 1){};

  \node[below left = of L1 0] (L2 0){};
  \node[below right= of L1 0] (L2 1){};
  \node[below right= of L1 1] (L2 2){};

  \node[below right = of L2 0] (L3 1){};
  \node[below right = of L2 1] (L3 2){};
   \node[below right = of L2 2] (L3 3){};
%
%
%


  \draw[thin arrow] (L0) -- (L1 0);
  \draw[thick arrow] (L0) -- (L1 1);

  \draw[thin arrow] (L1 0) -- (L2 0);
  \draw[thick arrow] (L1 0) -- (L2 1);
  \draw[thin arrow] (L1 1) -- (L2 1);
   \draw[thick arrow] (L1 1) -- (L2 2);

  \draw[thick arrow] (L2 0) -- (L3 1);
  \draw[thin arrow] (L2 1) -- (L3 1);
  \draw[thick arrow] (L2 1) -- (L3 2);
    \draw[thin arrow] (L2 2) -- (L3 2);
      \draw[thick arrow] (L2 2) -- (L3 3);
%

%
%
\end{tikzpicture} \captionof{figure}{Truncated crystal}\label{Figure 3}\end{minipage}
\end{figure}
Here the thin arrows indicate the first lowering Kashiwara operator $f_1$ and the thick arrows indicate the second $f_2.$
We call these crystals \emph{resonance families}. If $\bm$ is the resonant Lusztig datum that is the highest weight element of the associated resonance family, we write $RF(\bm)$ for this crystal graph. Proposition \ref{Prop: empty intersection} tells us that only those finitely-many families such that $RF(\bm)\cap\B(\lam+\rho)\neq \emptyset$ are truncated crystal graphs: the crystal is ``cut off'' at the intersection. See Theorem \ref{Thm: crystal structure} for a precise statement in terms of seminormality of the crystal graph structure on $RF(\bm)$ and the combinatorial decoration indicating which crystals are truncated.
 \begin{Rem}
 We note that this crystal structure is {not} derived from the Kashiwara crystal structure on $\B(-\infty)$ as these crystal graphs are contained in a single weight space of $\B(-\infty)$. 
 \end{Rem}
 This is a purely representation-theoretic statement about certain families of MV polytopes of type $G_2$ that are distinguished by these $p$-adic integrals. Note that the characterization of MV polytopes in terms of $\i$-Lusztig data \cite{Kam1} requires transcendental techniques: one must restrict to \emph{generic} elements of the affine Grassmanian to pass to valuations in order for the data to satisfy the tropical Pl\"{u}cker relations; this is the tropicalization functor mentioned in Remark \ref{Remark intro}. In contrast, the inability to do this in the $p$-adic setting isolates resonant MV polytopes as interesting objects of study. In particular, this crystal graph structure encodes the systematic cancellation needed to reduce the sum (\ref{eqn: whittaker sum}) to a finite sum.
 \begin{Thm}\label{Thm: external vanishing intro}
Let $RF(\bm)$ be the resonance family such that $RF(\bm)\cap\B(\lam+\rho)=\emptyset$. Then $RF(\bm)$ has the structure of a (full) Kashiwara crystal of type $A_1\times A_1$ and
\[
\sum_{\bn\in RF(\bm)}I_\lam(\bn)=0.
\]
\end{Thm}
 
 For those finitely many truncated crystals such that $RF(\bm)\cap\B(\lam+\rho)\neq \emptyset$, these contributions cannot be canceled. In this sense, the geometric Tokuyama-type formula for $G_2$ is not naturally a sum over $\B(\lam+\rho)$. One may attempt to incorporate these terms systematically to reduce to a sum over $\B(\lam+\rho)$, and in most cases there is a natural way to do this by appropriately augmenting the contributions of those Lusztig data $\bn\in RF(\bm)\cap\B(\lam+\rho)$. However, in the case of a \emph{totally resonant} resonance family, there does not appear to be a canonical way to do this. A family $RF(\bm)$ is totally resonant if $\bm=(a,b,c,b,a)$. 

For a $\lam$-relevant resonance family $RF(\bm)$, we set $RF(\bm)^\circ =RF(\bm)-(RF(\bm)\cap\B(\lam+\rho))$. We prove the following ``geometric'' \textbf{Tokuyama-type formula for $G_2$} (see Section \ref{Section: generalized tok} for all notation).
\begin{Thm}\label{Thm: tokuyamatype intro}
Let $\lam\in X^\ast(T)^+$ be a dominant weight for the split complex Lie group of type $G_2$ and let $\chi_\lam$ be its character. Then for $\tau\in T(\cc)$,
\begin{align*}
\prod_{\al\in \Phi^+}\left(1-q^{-1}\tau^\al\right)\chi_\lam(\tau)=\sum_{\bn\in \B(\lam+\rho)}&\tilde{G}(\bn)\tau^{\mu(\bn)-w_0(\lam+\rho)}\\
&+\sum_{\shortstack{$RF(\bm)$ \\$\lam$-$\mathrm{relevant}$\\{totally resonant}}}\left(\sum_{\bn\in RF(\bm)^\circ} G_{res}(\bn)\right)\tau^{\mu(\bn)-w_0(\lam+\rho)},
\end{align*}
where the contribution functions $\tilde{G}(\bn)$ and $G_{res}(\bn)$ are described in Section \ref{Section: generalized tok}.

\end{Thm}
\begin{Rem}
We note that this \emph{does not} recover the conjectured formula of \cite{friedlander2015crystal} (recently proven in \cite{defranco} with a combinatorial argument) term by term. Their formula was motivated by computer calculations, and it is not clear if there is a geometric interpretation of the augmentation to the standard contribution that they introduce.

\end{Rem} 
We expect the additional structure on the set of resonant MV polytopes in Theorem \ref{Thm: crystal intro} generalizes for arbitrary groups and reduced expressions $\i$. It would be interesting to find a purely representation-theoretic method of isolating this structure as well as which resonances are important. We end with a discussion of this in Section \ref{Section: resonance}.

\subsection{Acknowledgements}
I want to thank Solomon Friedberg for introducing me to questions which led directly to this project, as well as for many helpful conversations. I also wish to thank both Ben Brubaker and Dan Bump for helpful conversations on crystal graphs, Whittaker functions, and other topics on multiple occasions. Finally, I wish to thank the anonymous referee who explained to me the notion of (upper) seminormality and suggesting the correct statement of Theorem \ref{Thm: crystal structure}.

 \section{Canonical bases and Mirkovi\'{c}-Vilonen cycles}\label{Section: Prelim}

Let $\G$ be a simply-connected Chevalley group over a non-archimedean local field $F$. Then $\G$ arises as the $F$-points of a rational group scheme $\mathbb{G}$. We will have reason to consider the complex points of this group scheme, which we will denote by $\G_\cc$. Let $T\subset B^{\pm}\subset \G$ be a choice of split maximal torus $T$ and pair of opposite Borel subgroups $B^{\pm}$ with corresponding unipotent radicals $U^{\pm}$. We also choose these subgroups such that they arise as the $F$-points of $\qq$-sub-group schemes of $\mathbb{G}$. 
Let $X=X^\ast(T)$ be the lattice of characters, and let $Y=X_\ast(T)$ be the lattice of cocharacters of $T\cong F^\times\otimes_\zz Y$, with $\Lam_{++}$ the dominant Weyl chamber determined by $B$. For any $a\in F^\times$ and cocharacter $\mu\in \Lam$, we denote the image in $T$ by $a^\mu$.
Let $\Phi \subset X$ denote the associated reduced root system, $\Phi^+$ and $\Phi^-$ the corresponding sets of positive (resp., negative) roots, and $\Del$ the set of simple roots indexed by the finite set $I$. Let $\Phi^\vee\subset Y$ be the set of coroots, and let
\[
\la\cdot,\cdot\ra : X\times Y \lra \zz
\]
denote the canonical pairing induced by duality. Thus, for $\mu\in X$, and $\lam\in \Lam$, we have $\mu(a^\lam) = a^{\la\mu,\lam\ra}$.

Let $W$ denote the Weyl group associated to the root system $\Phi$. It is generated as a Coxeter group by simple reflections $s_i$, where $i\in I$. For an element $w\in W$, let $l(w)$ denote the length of $w$ as a reduced expression of the reflections $s_i$. Set $R(w)\subset I^{l(w)}$ to be the set of all reduced expressions of $w$: that is, for each $\i = (i_1,\ldots,i_l)\in R(w)$, we have $w=s_{i_1}\cdots s_{i_l}.$ Let $w_0\in W$ be the unique longest element of $W$, and set $l(w_0) = |\Phi^+|=N$. We refer to any $\i\in R(w_0)$ as a \emph{long word}. We denote by $\i^\ast=(i_1^\ast,\ldots,i_N^\ast)$ the dual long word, where $w_0(\al_i)=-\al_{i^\ast}$. As we will see, the choice of a long word $\i$ induces a linear order on the set of positive roots $\Phi^+$. In this context, $\Phi^+$ is in natural bijection with the indexing set $[N]=\{1,2,\ldots,N\}$.

For each simple root $\al_i\in\Del$, we fix an associated root subgroup embedding $x_i:\mathbb{G}_a(F)\lra \G$ such that $a^\lam x_i(t)a^{-\lam}= x_i(a^{\la\al_i,\lam\ra}b)$, for all $a\in F^\times$ and $b\in F$. This pinning induces a unique Chevalley embedding $\varphi_i: \SL_2(F)\lra \G$ such that
\[
\varphi_i\left(\begin{array}{cc}1&b\\ &1\end{array}\right) = x_i(b)\qquad \mbox{and} \qquad\varphi_i\left(\begin{array}{cc}a&\\ &a^{-1}\end{array}\right) = a^{\al_i^\vee}.
\]
Set now 
\[
y_i(b) = \varphi_i\left(\begin{array}{cc}1&\\ b&1\end{array}\right) \qquad \mbox{and}  \qquad \overline{s_i} = \varphi_i\left(\begin{array}{cc}&-1\\1 &\end{array}\right).
\]

The elements $\overline{s_i}$ generate a subgroup of $\G$, $\widetilde{W}$, which is a finite cover of the Weyl group, $W$ such that $\overline{s_i}\mapsto s_i$. For any $w\in W$, and for any $\i\in R(w)$, we have that the element $\overline{w} = \overline{s_{i_1}}\cdots\overline{s_{i_l}}\mapsto w$, where $l(w) = l$. Note that the elements $\overline{s_i}$ satisfy the appropriate braid relations, so that the lift $\overline{w}$ does not depend on the choice of word $\i\in R(w)$.

For any positive root $\al\in\Phi^+$, choose a simple root $\al_i$ and a Weyl group element $w\in W$ such that $\al = w\cdot \al_i$. Then we may define the subgroups
\begin{equation}\label{eqn: root groups}
x_\al(b)= \overline{w}x_i(b)\overline{w}^{-1} \qquad \mbox{and} \qquad y_{\al}(b) = \overline{w}x_i(b)\overline{w}^{-1},
\end{equation}
and set $\overline{s_{\al}} = \overline{w} \overline{s_i}\overline{w}^{-1}$. Note that this choice differs from the pinning in \cite{McN}, and is responsible for the positivity statements in e.g. Proposition \ref{Prop: character sum}.
\subsection{Relations}
We recall here the relations between the various root embeddings defined in the previous section. These are standard for Chevalley groups \cite{Springerbook}.
\begin{itemize}
\item For all $\al,\be\in \Phi$, 
\begin{equation}\label{eqn: relation 1}
a^{\be^\vee}x_\al(b)a^{-\be^\vee} = x_\al(a^{\la\al,\be^\vee\ra}b), \qquad a^{\be^\vee}y_\al(b)a^{-\be^\vee} = y_\al(a^{-\la\al,\be^\vee\ra}b).
\end{equation}
\item Suppose $\al,\be\in \Phi$ with $\al+\be\neq 0$. Then
\begin{equation}\label{eqn: relation 2}
x_\al(a)x_\be(b)x_\al(a)^{-1}= x_\be(b)\left[\prod_{\shortstack{$\scriptstyle i,j\in \zz_{>0}$\\ $\scriptstyle \al+j\be=\ga\in \Phi$}}x_\ga(c_{\al,\be;\i,j}a^ib^j)\right].
\end{equation}
 Our notation requires we write out several cases of this relation depending on the positivity of the roots $\al$ and $\be$; we omit the other cases as their form is identical with $x_\al(b)$ replaced with $y_\al$ and $\al$ replaced with $-\al$.
\end{itemize}
\subsection{Involutions}

Define the ``transpose'' involutive anti-automorphism $x\mapsto x^T$ of $G$ to be
\[
t^T= t,\quad\text{for }t\in T,\qquad x_i(a)^T=y_i(a).
\]
Next, we define the ``positive twist'' to be the involutive anti-automorphism $x\mapsto x^\iota$
\[
t^\iota= t^{-1},\quad\text{for }t\in T,\qquad x_i(a)^\iota=x_i(a).
\]

\subsection{Gauss Decompositions}
Our choice of Borel $B^+$ naturally gives rise to two opposite Bruhat decompositions
\[
\G = \bigsqcup_{w\in W} B^+wB^+ = \bigsqcup_{w\in W} B^-wB^-
\]
On the Zariski-open subset $G_0:=B^+B^-=U^+TU^-$, every element $g$ has a \emph{unique} Gauss decomposition $g=utn$, where $u\in U^+,$ $t\in T$, and $n\in U^-$  . We adopt the notation from \cite{BZ01} and write $u=[g]_+$, $t=[g]_0$, and  $n=[g]_-$.

\subsection{Lusztig's parametrizations of the canonical basis}\label{Sec: canonical}
 Let $\fg$ be the simple complex Lie algebra of the same Cartan type as $G$, and fix a Cartan subalgebra $\fh$. For example, $\fg = \cc\otimes_\qq \fg_0$, where $\fg_0$ is the split rational Lie algebra associated to a Chevalley group. Let $$U_q(\fg) = U_q(\fu^+)\otimes U_q(\fh)\otimes U_q(\fu^-)$$ be the associated quantized universal enveloping algebra, where $\fu^-=\Lie(U^-)$  and $\fu^+=\Lie(U^+)$. One can write down relations between the generators $\{E_i,F_i, K_i\}_{i\in \Delta}$, where $\{\al_i\}_{i\in \Del}$ is the set of simple roots; see \cite{BZ01}.
 
To each $i\in I$, Lusztig associates an algebra automorphism $T_i$ of $U_q(\fg)$. These operators satisfy the braid relations, inducing an action of the braid group on $U_q(\fg)$. Set now $U_q^+= U_q(\fu^+)$ to the be subalgebra generated by the $E_i$. 

\begin{Def}
Fix $\i=(i_1\ldots, i_N)\in R(w_0)$. To each $N$-tuple $\bm=(m_1,\ldots,m_N)\in \zz^N_{\geq 0}$, consider the element
\[
p_{\i}^{(\bm)}:=E_{i_1}^{(m_1)}T_{i_1}(E_{i_2}^{(m_2)})\cdots (T_{i_1}\cdots T_{i_{N-1}}(E_{i_N}^{(m_N)})).
\]
It is shown in \cite[Cor. 40.2.2]{Luzbook} that the collection $\mathcal{B}_{\i}$ of such elements forms a $\cc(q)$-basis for $U_q^+$. This is referred to as the \emph{PBW-type basis} corresponding to $\i$.
\end{Def}

The canonical basis of $U_q^+$ is defined in terms of these PBW-type bases. Denote by $u\mapsto \overline{u}$ the $\cc$-linear involutive algebra anti-automorphism determined by $\overline{q}= q^{-1}$ and $\overline{E_i}=E_i$.

\begin{Prop}
For every $\i\in R(w_0)$ and $t\in \zz^N_{\geq0}$, there is a unique element $b=b_{\i}(\bm)$ of $U_q^+$ such that $\overline{b} = b$, and $b-p_{\i}^{(\bm)}$ is a linear combination of elements of $\mathcal{B}_{\i}$ with coefficients in $q^{-1}\zz[q^{-1}]$. For a fixed $\i$, the elements $b_{\i}(\bm)$ constitute the canonical basis $\mathcal{B}(-\infty)$ of $U_q^+$.
\end{Prop}
In this way, the choice of $\i$ induces a bijection $b_{\i} :\zz^N_{\geq0}\lra \mathcal{B}(-\infty)$, which is called the $\i$-Lusztig data parametrization of $\B(-\infty)$. When $\i$ is clear from the context, we will omit it from the terminology. 

\subsection{Toric charts on the unipotent radical}\label{Section: toric}
In this section, we recall certain embeddings of algebraic tori in unipotent radicals. The toric charts play a central role in the study of double Bruhat cells, as in \cite{BZ97} and \cite{BZ01}, as they tropicalize to the various combinatorial parametrizations of canonical bases.
In general, given a pair $u,v\in W$ of Weyl group elements, the intersections $$G^{v,u} =B^-vB^-\cap B^+uB^+,$$ referred to as \emph{double Bruhat cells}, carry important positive structures. In particular, the subsets
\[
L^{v,u} = U^-\overline{v}U^-\cap B^+uB^+,
\] 
known as \emph{reduced double Bruhat cells}, are isolated as the appropriate objects of study in \cite{BZ01}. We consider only the case $L^{e,w_0} = U^-\cap B^+w_0B^+$. 
Explicitly, for a fixed long word $\i\in R(w_0)$, consider the embeddings $y_{\i}:\mathbb{G}_m^N \lra L^{e,w_0}$, given be
\[
y_{\i}(t_1,\ldots,t_N) =y_{i_1}(t_1)\cdots y_{i_N}(t_N).
\]
\ 
For any fixed long word $\i$, we may identify the tropical points $L^{e,w_0}(\zz^t)$ with a lattice $\zz^{|\Phi^+|}$, which naturally contains the cone of $\i$-Lusztig data as the first orthant \cite{BZ01}.


\subsection{Generalized minors}
One of the main tools developed by Berenstein, Fomin, and Zelevinsky to study the positive structures on double Bruhat cells and their tropicalizations is the theory of generalized minors. We recall the definitions and properties we will need, closely following \cite[Section 4]{BZ01}. For $i\in I$, denote by $\Lam_i$ the corresponding fundamental weight. For any $u,v\in W$, the generalized minor $\Del_{u\Lam_i,v\Lam_i}$ is the regular function on $G$ whose restriction to the open subset $\overline{u}G_0\overline{v}^{-1}$ is given by 
\[
\Del_{u\Lam_i,v\Lam_i}(x) = ([\overline{u}^{-1}x\overline{v}]_0)^{\Lam_i}.
\]
That is, $\Del_{u\Lam_i,v\Lam_i}$ is the matrix coefficient in the fundamental representation $V(\Lam_i)$ between the associated extremal weight spaces. It is known that $\Del_{u\Lam_i,v\Lam_i}$ depends only on the weights $u\Lam_i$ and $v\Lam_i$, and not the choice of representatives.

These functions satisfy many properties, which we recall now. Fix weights $\ga, \de\in W\Lam_i$.
\begin{itemize}
\item For elements of the torus $t_1,t_2\in T$, we have
\begin{equation*}
\Del_{\ga,\de}(t_1xt_2) = t_1^\ga t_2^\de\Del_{\ga,\de}(x).
\end{equation*}
\item Generalized minors behave well with respect to the main involutions on $G$:
\begin{equation*}\label{eqn: involutions}
\Del_{\ga,\de}(x) =\Del_{\de,\ga}(x^T) = \Del_{-\de,-\ga}(x^\iota) =\Del_{w_0\de,w_0\ga}(\tau(x))
\end{equation*}
\end{itemize}
While we do not recall the formulas in full generality, we remark that these functions satisfy the so-called generalized Pl\"{u}cker relations. In particular, of $l(us_i) =l(u)+1$ and $l(vs_i)=\l(v)+1$, then these satisfy the following \emph{exchange relation}
\begin{equation*}
\Del_{u\Lam_i,v\Lam_i}\Del_{us_i\Lam_i,vs_i\Lam_i}=\Del_{us_i\Lam_i,v\Lam_i}\Del_{u\Lam_i,vs_i\Lam_i}+\prod_{j\neq i}\Del_{u\Lam_j,v\Lam_j}^{-a_{ji}}
\end{equation*}
This relation may be used to show that generalized minors induce a cluster algebra structure on the rings of regular functions on double Bruhat cells \cite{BFZ}. 

 To evaluate generalized minors on unipotent elements of the form $y_{\i}(b_\bullet)$, we recall the notion of $\i$-trails from \cite{BZ01}.
\begin{Def}
Let $V$ be a finite-dimensional $\fg$-module, $\ga$ and $\delta$ are two weights in $P(V)$, and $\i=(i_1,\ldots, i_l)\in R(w)$ for some $w\in W$. An {\bf$\i$-trail} from $\ga$ to $\de$ in $V$ is a sequence of weights $\pi=(\ga=\ga_0,\ga_1,\ldots, \ga_l=\de)$ such that 
\begin{enumerate}
\item for $1\leq k\leq l$, $\ga_{k-1}-\ga_k = c_k(\pi)\al_{i_k}$ for some non-negative integer $c_k(\pi)$,
\item $e_{i_1}^{c_1(\pi)}\cdots e_{i_l}^{c_l(\pi)}$ is a non-trivial linear map from $V(\de)$ to $V(\ga)$.
\end{enumerate}

We use the notation $\pi: \ga\lra \de$ to indicate that $\pi$ is an $\i$-trail from $\ga$ to $\de$.
\end{Def}
The notion of $\i$-trails  was introduced in \cite{BZ01} to study canonical bases via geometric liftings. In particular, they encode certain generalized minors.

\begin{Thm}\cite[Theorem 5.8]{BZ01}\label{Thm 5.8}
Let $\ga$ and $\de$ be two weights in the $W$-orbit of the same fundamental weight $\Lam_i$ of $\fg$, and let $\i=(i_1,\ldots,i_N)$ be any sequence of indices in $\{1,2,\ldots,r\}$. Then $ \Del_{\ga,\de}(x_{\i}(t_1,\ldots,t_N))$ is a positive integer linear combination of the monomials $t_1^{c_1(\pi)}\cdots t_N^{c_N(\pi)}$ for all $\i$-trails from $\ga$ to $\de$ in $V_{\Lam_i}$.
\end{Thm}

In particular, there exist positive integers $d_\pi$ such that
\begin{equation}\label{equation: trails}
\Del_{\ga,\de}(x_{\i}(t_1,\ldots,t_N))=\sum_{\pi:\ga\rightarrow \de}d_\pi t_1^{c_1(\pi)}\cdots t_N^{c_N(\pi)}.
\end{equation}
In general, a formula for these coefficients is not known; they are built out of intricate applications of the commutation relations (\ref{eqn: relation 2}). 


\subsection{Parametrizations of Mirkovic-Vilonen polytopes}
Let $\mathcal{K}=\cc((t))$ be the field of Laurent series with coefficients in $\cc$, and $\calo=\cc[[t]]$ is the ring of power series. Consider the affine Grassmanian $\mathcal{G}r=G(\mathcal{K})/G(\calo)$ with its analytic topology.  
In proving the geometric Satake equivalence \cite{MV07}, Mirkovi\'{c} and Vilonen identify cycles (the MV cycles) of $\mathcal{G}r$ which induce a weight basis in a uniform manner in each rational representation of $\mathbb{G}^\vee$. This geometrically-defined basis is the MV basis. In \cite{Kam1}, Kamnitzer showed how to determine these cycles (or rather free $Y$-orbits of cycles known as stable MV cycles) from the combinatorics of their moment polytopes: these are the MV polytopes. An important step in the proof requires fixing a long word $\i\in R(w_0)$ and parametrizing semi-infinite cells with certain toric charts. By passing to valuations, these toric charts encode $\i$-Lusztig data of a canonical basis element, inducing a bijection between Lusztig's canonical basis $\B(-\infty)$ and the collection of (stable) MV-cycles.

We recall these parametrizations, which depends on the twist map of Berenstein and Zelevinsky \cite{BZ97}. This twist $\eta :L^{e,w_0} \lra L^{e,w_0}$ is the birational map given by 
\[
\eta(u) = [\overline{w_0}u^T]_-.
\]

 It is not hard to check that $\eta^{-1}= \iota\circ\eta\circ\iota$. In particular, if we define the map $\psi_{w_0}=\eta\circ\iota$, then $\psi_{w_0}$ is an involution on $L^{e,w_0}$, which we refer to as the BZ-involution. We now fix a long word $\i$, and recall $l(w_0)=|\Phi^+|=N$. Denote by $z_{\i}(b_\bullet)$ the element 
\[
z_{\i}(b_\bullet):= \psi_{w_0}(y_{\i}(b_1,\cdots, b_N)).
\]

We now fix a long word $\i$, and recall $l(w_0)=|\Phi^+|=N$. In \cite{Kam1}, Kamnitzer defines subsets $A^{\i}(n_\bullet)$, where $(n_\bullet)\in \zz^N_{\geq0}$, in terms of lengths of edges along a certain path (determined by $\i$) in the $1$-skeleton of certain polytopes; these lengths $(n_\bullet)$ are called the $\i$-Lusztig data of the polytope. While this description of $A^{\i}(n_\bullet)$ is crucial to understanding MV-cycles, the following characterization is more important for our purposes.

\begin{Thm}\cite[Theorem 4.5, Theorem 3.1]{Kam1}\label{Thm: An}
With $\i$ and $(n_\bullet)\in \zz^N_{\geq0}$ as above, we have
\[
A^{\i}(n_\bullet)= \{[z_{\i^\ast}(b_\bullet)] \in \mathcal{G}r : b_k\in \mathcal{K}^\times,\;\val(b_k) = n_k\}.\footnote{The dual word arises here due to our definition of the twist map. This choice is motivated by later computations related to the Iwasawa decomposition.}
\]
Moreover, the Zariski-closure $Z^{\i}(n_\bullet) =\overline{A^{\i}(n_\bullet)}$ is an MV-cycle, and each stable MV-cycle has a unique representative cycle (the one of coweight $(0,\mu)$ for some $\mu$) that arises in this way.
\end{Thm}
 In this way, one defines the notion of the $\i$-Lusztig data of a stable MV-cycle. Letting $\mathcal{L}(-\infty)$ denote the collection of stable MV cycles, Kamnitzer shows that this gives a bijection
\begin{align*}
\mathcal{L}(-\infty) &\longleftrightarrow \:\:\B(-\infty)\\ Z^{\i}(n_\bullet)\quad&\longleftrightarrow b_{\i}(n_\bullet),
\end{align*}
where $b_{\i}:\zz^N_{\geq0}\lra \B$ was the $\i$-Lusztig parametrization as in Section \ref{Sec: canonical}.

\begin{Rem}
With this identification, once a long word $\i$ is selected, we frequently use Kamnitzer's result to identify a crystal graph $\B(\lam)$ or $\B(-\infty)$  with the corresponding set of MV cycles or MV polytopes $\mathcal{L}(\lam)$ or $\mathcal{L}(-\infty)$ without comment.
\end{Rem}

\subsection{Transition to non-archimedean setting}
Several of the results and coordinates of the preceding sections were obtained for reductive complex algebraic groups, and we pause to justify any use of these results when dealing with groups defined over $\qq$ or $\qq_p$. The point is that we only apply decompositions when the corresponding functions are known to be rational with $\zz$-coefficients. 
Additionally, many of the maps and constructions used in the previous section are only birational. This is nevertheless suitable for our purposes due to the following two easy lemmas.
\begin{Lem}\label{Lem: Zariski topology}
Suppose $V$ is a quasi-affine algebraic variety over $\qq$, with rational points $V_\qq$ and complex points $V_\cc$. Then the subspace topology on $V_\qq$ under the inclusion $V_\qq\subset V_\cc$ agrees with the Zariski topology on $V_\qq$.
\end{Lem}

\begin{Lem}\label{Lem: measure zero}
Let $F$ be a locally-compact topological field of infinite cardinality, and let $\mu$ denote the Haar measure on $F^n$. If a subset $V\subset F^n$ is Zariski-closed (identifying $F^n = \mathbb{A}^n_F$), then $\mu(V) = 0$.
\end{Lem}

We now see how these results enable us to use certain results from \cite{BZ01}: for any $\i\in R(w_0)$, the image $U^-_{\i}$ of the toric chart $y_{\i}$ is Zariski-open (hence, dense by irreducibility) in $L^{e,w_0}_\cc$. In fact, it follows from \cite[Theorem 4.4]{Kam1} that the polynomials cutting out the closed compliment are all defined over $\zz$. By Lemma \ref{Lem: Zariski topology}, we see that $U^-_{\i}$ is also Zariski-open over $\qq$, hence over any extension $F$ as base change preserves open immersions. 

Since $L^{e,w_0}$ is Zariski-open and dense in $U^-$, so is $U^-_{\i}$ . Lemma \ref{Lem: measure zero} now implies that the compliment of $U^-_{\i}\subset U^-$ has measure zero. Therefore, we may replace integration over $U^-$:
\begin{equation}\label{eqn: functional identity}
\int_{U^-} = \int_{U^-_{\i}}.
\end{equation}

\section{Spherical Whittaker functions}\label{Section: Whittaker}

Let us now assume that $F$ is a local non-archimedean field with residue characteristic $p$ containing a full $n^{th}$ set of roots of unity. For simplicity, we will assume that $\mu_{2n}\subset F$ and that $(p,n)=1$. Let $\calo$ denote the ring of integers in $F$, let $\vp$ be a fixed uniformizer, and let $|\calo/\vp|=q$. 
In this section, we recall the notion of a spherical Whittaker function of an unramified principal representation of $\Gn$, where $\Gn$ is an $n$-fold cover of our fixed simply-connected Chevalley group. We then recall the explicit algorithm for an Iwasawa decomposition of \cite{McN} and relate it to the toric charts on $U^-$. Motivated by this reinterpretation, we will augment the algorithm to make the task of expressing the MV integrals in terms of these coordinates more algebraic.

\subsection{Principal series representations and spherical Whittaker vectors}
We refer to \cite{leslie2017generalized} for the relevant theory of covering groups and principal series; all notation used here matches what is introduced there. 
We assume that we have a fixed degree $n$ cover
\[
1\lra \mu_n\lra \Gn\xrightarrow{p}\G\lra 1.
\]
Denote the hyperspecial maximal compact subgroup $K=\G(\calo)$, $H_K=H\cap K$ for any closed subgroup $H\subset \G$, and set $\overline{H}=p^{-1}(H)$.  Assume that $\chi: Z(\overline{T})T_K\lra \cc^\times$ is an unramified character, and let $I(\chi)$ be the associated unramified principal series representation. This is defined by first inducing to $\overline{T}$
\[
i(\chi) = \ind_{Z(\overline{T})T_K}^{\overline{T}}(\chi),
\] 
followed by parabolically inducing from $\overline{B}$ to $\Gn$. There exists a unique $K$-fixed line in $I(\chi)$, where $K$ acts through our fixed splitting $\kappa: K\lra \Gn$:

\begin{Thm}\cite[Theorem 5.1]{McN}
The unramified representation $I(\chi)$ has a one-dimensional space of $K$-fixed vectors.
\end{Thm}

Let $\varphi_K:\Gn\lra i(\chi)$ be a non-zero vector in this space. We are interested in computing Whittaker functions associated to this vector. When $n>1$, the dimension of the space of Whittaker functionals $Wh(\chi)$ on $I(\chi)$ is $\dim(i(\chi))$. To encode the choice of Whittaker functional, we adopt the notation of \cite{McN} by choosing for each $i\in I$ a complex number $x_i$ such that 
\[
\chi\left(\prod_{i}\overline{h}_{\al_i}(\varpi^{m_i})\right) = \prod_{i}x_i^{m_i}
\]
whenever the argument lies in the group $Z(\overline{T})T_K$ and $\sum_im_i\al_i^\vee\in Y^{sc}_{Q,n}$ (a certain sublattice of $Y$ determined by the covering group).  This choice is equivalent to fixing an isomorphism 
\[
\Ind_{\overline{B}}^{\Gn}\left(i(\chi)\right) \cong \Ind_{Z(\overline{T})T_K}^{\Gn}(\chi\otimes \de^{1/2}),
\]
where $\de$ is the modular character of $\overline{B}$. With this choice, we obtain an associated function $f:\Gn\lra \cc$ from the spherical vector $\varphi_K$ given by
\[
f\left(\zeta u\prod_{i}\overline{h}_{\al_i}(\varpi^{m_i})k\right) = \zeta\prod_{i}(q^{-1}x_i)^{m_i}
\]
where $\zeta\in \mu_n$, $u\in U$, $m_i\in\zz$, and $k\in K$. This suffices to define $f$ by the Iwasawa decomposition. 

Fix an additive character $\psi$ of $F$ of conductor $\calo$ and denote also by $\psi$ the induced generic character on $U$. That is, we require that for each simple root $\al$, the restriction of $\psi$ to the root group $U_{\al}\cong F$ is exactly $\psi$. We define the Whittaker function associated to $\varphi_K$ (or equivalently to $f$) by the integral
\[
 \int_Uf(\overline{w_0}^{-1}ug)\psi(u)du.
\] 
By the Iwasawa decomposition and the $K$-invariance of $f$, we may rewrite this as
\[
W_\chi(\lam) =  \int_{U^-}f(u\vp^{w_0\lam})\psi(u)du.
\]
Conjugating the torus element past $u$ and changing variables, we obtain (up to a shift) the integral 
\[
I_\lam = \int_{U^-}f(u)\psi_\lam(u)du
\]
where $\psi_\lam(u) = \psi(\vp^{w_0\lam}u\vp^{-w_0\lam})$. It is a standard fact \cite[Lemma 5.1]{CassShalika} that $I_\lam=0$ unless $\lam$ is dominant. 

\subsection{An explicit algorithm for the Iwasawa decomposition}\label{Section: algorithm}
For the purpose of computing such Whittaker functions, McNamara introduces in \cite{McN} an explicit algorithm for an Iwasawa decomposition for an element $u\in U^-$: for a fixed long word $\i\in R(w_0)$, this algorithm writes $u = p_1k$, where $p_1\in B$ and $k\in K$. In this section, we recall this algorithm and show that, under certain assumptions on the element $u\in U^-$, we can relate this algorithm to the toric charts in Section \ref{Section: toric}.

Let us recall the algorithm: let ${\i}= (i_1,\ldots, i_N)\in R(w_0)$ be a fixed long word; this induces a convex total ordering $<_{\i}$ on $\Phi^+$:
\[
\Phi^+=\{\ga_1<_{\i}\cdots<_{\i} \ga_N\}, \quad \ga_j=s_{i_N}s_{i_{N-1}}\cdots s_{i_{j+1}}\al_{i_j}.
\]

Define $G_k$ to be the subset of elements in $G$ which can be written in the form
\[
g = \left(\prod_{j=1}^{k-1}y_{\ga_j}(x_j)\right)t\left(\prod_{j=k}^{N}y_{\ga_j}(x_j)\right),
\]
where $t\in T$.
\subsubsection{Algorithm}\label{section: Algorithm}
Begin with $u = \prod_{j=1}^N y_{\ga_j}(x_j)$, and set $p_{N+1}=1$. In decreasing order of $k$, McNamara shows inductively that there exists $p_k'\in G_k$ and $t_k\in F$ such that $y_{\ga_k}(x_k)p_{k+1} = p_k'y_{\ga_k}(t_k)$.  Now set 
\[
p_k = \begin{cases} p_k'  \qquad\qquad\qquad\qquad\qquad\;\mbox{if  } &t_k\in \calo\\ p_k't_k^{-\ga_k}x_{\ga_k}(t_k) \qquad\quad\qquad\;\mbox{ if} &t_k\notin \calo.\end{cases}
\]
The algorithm terminates with $p_1$, and clearly one has $u=p_1k$ with $k\in K$ \cite[Theorem 4.5]{McN}.

The algorithm produces the variables 
\begin{equation}\label{eqn: w coords}
w_k=\begin{cases} t_k \qquad \mbox{if } |t_k|>1\\ 1\qquad \mbox{ if } |t_k|\leq1\end{cases},
\end{equation}
so that if we write $p_1 = h\cdot u$, with $h\in T$ then 
\[
h=\prod_{j=N}^1w_j^{-\ga_j}.
\]
Setting $m_k(u;\i) =-\val(w_k)\geq 0$, define for any ${\bf m}=(m_1,\ldots,m_N)\in \zz_{\geq0}^N$
\begin{equation*}\label{eqn: unipotent cell}
C^{\i}(\bm) = \{ u\in U^- : m_k(u;\i) =m_k\;\;\mbox{for all } k\in [N]\}.
\end{equation*}
We obtain a decomposition
\begin{equation}\label{eqn: decomposition}
U^-=\bigsqcup_{{\bf m}\in \zz_{\geq0}^N}C^{\i}(\bf m).
\end{equation}

In \cite[Theorem 7.2]{McN}, McNamara shows that this decomposition is essentially the same as the one arising in \cite{Kam1}. In particular, the integers ${\bf m}\in \zz_{\geq0}^N$ become identified with $\i$-Lusztig data for the corresponding MV polytope, so that the decomposition (\ref{eqn: decomposition}) writes $U^-$ as a disjoint union indexed by MV polytopes, or the crystal graph $\B(-\infty)$.  
The theorem below states that the coordinates $\{t_\al,w_\al\}$ may be explicitly related to the toric charts employed by Kamnitzer. This allows us to apply results of \cite{BZ01} on these charts.
\begin{Thm}\label{Thm: geom algo}
Fix $(b_1,\ldots,b_N)\in (F^\times)^N$, and define such that
$$u=z_{\i}(b_\bullet)= \psi_{w_0}\left(y_{\i}(b_1,\ldots, b_N)\right)\in C^{\i}(\bm),$$
 If $m_k>0$ for all $k\in [N]$, then $w_k = \frac{1}{b_k}$ for each $j$, where $\{w_j\}$ are the coordinates (\ref{eqn: w coords}) arising from the Iwasawa decomposition algorithm. In particular, $\val(b_k) = m_k$.

\end{Thm}
\begin{proof}
The assumption on the $\i$-Lusztig data ${\bf m}$ implies that at each step of the algorithm, $p_j =p_j't_j^{-\ga_j}x_{\ga_j}(t_j)$. As a result, we may identify the element $k\in K$ in $u=p_1k$ as
\[
k = \overline{s_{\ga_1}}x_{\ga_1}(t_1^{-1})\overline{s_{\ga_2}}x_{\ga_2}(t_2^{-1})\cdots\overline{s_{\ga_N}}x_{\ga_N}(t_N^{-1})
\]
Note that since we have $\ga_j =s_{i_N}s_{i_{N-1}}\cdots s_{i_{j+1}}\al_{i_j},$ it follows that 
\[
\overline{s_{\ga_j}} = \overline{s_{i_N}}\;\overline{s_{i_{N-1}}}\cdots \overline{s_{i_{j+1}}}\;\overline{s_{i_j}}\;\overline{s_{i_{j+1}}}^{-1}\cdots \overline{s_{i_N}}^{-1}.
\]
By the formulas in (\ref{eqn: root groups}), one may easily check that this implies 
\[
k = \overline{w_0} x_{i_1}(t_1^{-1})x_{i_2}(t_2^{-1})\cdots x_{i_N}(t_N^{-1}) = \overline{w_0}x_{\i}(t_1^{-1},\ldots,t_N^{-1}),
\]
where we have used the identification $\overline{s_{\ga_1}}\;\overline{s_{\ga_2}}\cdots\overline{s_{\ga_N}}=\overline{w_0}$. By our assumption on the Lusztig data, $t_\al=w_\al$, so that 
\begin{equation*}
x_{\i}(w_1^{-1},\ldots,w_N^{-1})=[(\overline{w_0}^{-1}u)^\iota]_+^\iota.
\end{equation*}
On the other hand, we have $y_{\i}(b_1,\ldots, b_N)=\psi_{w_0}(u) =[\overline{w_0}u^{\iota T}]_-$. Note that for any $u\in L^{e,w_0}$, there exist $z,v\in U^+$ and $t\in T$ such that $u=tz\overline{w_0}v$.
One then checks that $[(\overline{w_0}^{-1}u)^\iota]_+^\iota = v$, and that $[\overline{w_0}u^{\iota T}]_- = v^{\iota T}$.

Therefore, $x_{\i}(w_1^{-1},\ldots,w_N^{-1})=y_{\i}(b_1,\ldots, b_N)^{\iota T}$. By injectivity of the maps $x_{\i}$ and $y_{\i}$, we conclude $b_j=w_j^{-1}$ for each $j\in [N]$.
\end{proof}

We end this section by augmenting the algorithm presented above to take advantage of the conceptual consequences of the above theorem. 
\begin{Algo}\label{Algo 2}
Beginning with $u = \prod_{j=1}^N y_{\ga_j}(x_j)$, and set $p_{N+1}=1$. In decreasing order of $k$, there exists $p_k'\in G_k$ and $t_k\in F$ such that $y_{\ga_k}(x_k)p_{k+1} = p_k'y_{\ga_k}(t_k)$.  Now set 
\begin{equation*}
w_k=\begin{cases} t_k \qquad \mbox{if } |t_k|>1\\ 1\qquad \mbox{ if } |t_k|\leq1\end{cases},
\end{equation*}
as before. Now set
\[
p_k = p_k'w_k^{-\ga_k}x_{\ga_k}(w_k).
\]
The algorithm terminates with $p_1$, and clearly one has $u=p_1k$ with $k\in K$.\qed
\end{Algo}
This removes the issue, notably present in the proof of Theorem 8.4 of \cite{McN}, that certain terms in the additive character $\psi_\lam(u)$ can vanish in subtle ways depending on the different cases of the algorithm. We state this as the following proposition.
\begin{Prop}\label{Prop: independence}
For any $u\in U^-$, write $u= \prod_{j=1}^N y_{\ga_j}(x_j)$. For any $k\in \{1,\ldots, N\}$, write $x_k=h_k(t_\al,w_\al)$, where $h_k(\underline{x},\underline{y})\in \zz[x_1,\ldots,x_N,y_1^{\pm1},\ldots,y_N^{\pm1}]$ and $(t_\al)$ and $(w_\al)$ are the variables arising from Algorithm \ref{Algo 2}. Then the Laurent polynomial $h_k(\underline{x},\underline{y})$ is independent of $u$.
\end{Prop}
\begin{proof}  The fact that $h_k$ is a Laurent polynomial follows from the group relations (\ref{eqn: relation 1}) and (\ref{eqn: relation 2}). The independence follows from our adoption of the augmented Algorithm \ref{Algo 2}, which is designed so that for each $k$, the set 
\[
y_{\ga_k}(x_k)p_{k+1} = p_k'y_{\ga_k}(t_k)
\] 
requires performing the same set of conjugations (\ref{eqn: relation 1}) and (\ref{eqn: relation 2}) independent of the valuations of $t_j$ for $j>k$. The coefficients of the monomials and degrees are thus determined by these relations and independent of the cases.
\end{proof}

Changing the algorithm in this way translates the problem of computing the Iwasawa decomposition into the algebraic problem of computing the Laurent polynomials $h_i(\underline{x},\underline{y})$. In the next section, we use the theory of generalized minors to compute the specializations $h_i(\underline{x},\underline{x})$ (corresponding to those MV cycles where the $\i$-Lusztig data are positive) for a general long word $\i$. For certain words, this is sufficient to compute the entire Laurent polynomial $h_i(\underline{x},\underline{y})$. We discuss this problem at the end of the next section.
 
\section{MV integrals in terms of Lusztig data}\label{Section: Application}

For each dominant coweight $\lam\in \Lam^+=X_\ast(T)^+$, we seek to compute the integral
\[
I_\lam = \int_{U^-}f(u)\psi_\lam(u)du.
\]
Combining this with the decomposition (\ref{eqn: decomposition}), we may write
$
I_\lam = \sum_{\bm}I_\lam(\bm),
$
where 
\begin{equation}\label{eqn: MV integral}
I_\lam(\bm) = \int_{C^{\i}(\bm)}f(u)\psi_\lam(u)du.
\end{equation}

Due to the connection to MV polytopes and canonical bases, we refer to $I_\lam(\bm)$ as the MV integral associated to $(\i,\bm)\in \B(-\infty)$. It is desirable to compute these integrals in terms of the coordinates $\{t_\al\}$ produced by the algorithm. This was accomplished for $\G=\SL_{r+1}$ with respect to the Gelfand-Tsetlin long word in \cite{McN}, recovering the metaplectic Tokuyama formula of \cite{BBF1} with remarkable efficiency.
\subsection{Expressing $I_\lam(\bm)$ in terms of the algorithm}
 Fix a long word $\i$ as well as an $\i$-Lusztig datum $\bm\in \zz^N_{\geq0}$. As noted in the introduction, the spherical function $f$ is constant on the cell $C^{\i}(\bm)$ with known value. We recall the formula here.
\begin{Prop}\cite[Lemma 6.3 (2)]{McN}\label{Prop: f const}
The spherical function $f$ is constant on $C^{\i}(\bm)$ with value
\[
\prod_{\al\in \Phi^+}\left(q^{-\la\rho,\al^\vee\ra}x_\al\right)^{m_\al},
\]
where the variables $x_\al=\prod_ix_i^{h_i}$, with $\al^\vee=\sum_ih_i\al^\vee_i$. Moreover, if we set $w_\al = \vp^{-m_\al}u_\al$, then
\[
f(u)= \prod_{\al\in \Phi^+}\left(q^{-\la\rho,\al^\vee\ra}x_\al\right)^{m_\al}(u_\al,\vp)^{m_\al+\sum_{\al<_{\i}\be}\la\be,\al^\vee\ra m_\be}.
\] \qedhere
\end{Prop}
As we are making use of the identity (\ref{eqn: functional identity}), we pause here to consider the measure relation between the Haar measure $du$ and the natural toric measure on $U^-_{\i}$.  While the full result will not be used in the following, it has the nice consequence that the nonarchimedean Whittaker function may naturally be expressed as a sum of the tropical points of a \emph{geometric crystal}, though we do not consider this expansion here.
\begin{Prop}\label{Prop: measures}
Let $U^-_{w_0}$ be the (open, dense) intersection of the image of the toric charts $y_{\i}$, and let $du$  denote the Haar measure on $U^-$ normalized so that $\mathrm{vol}(U^-(\calo))=1$. Then, we have that for $u= y_{\i}(b_\bullet)$
\begin{equation*}\label{measure change}
du = \prod_{\al\in\Phi^+}|b_\al|^{-\la \rho,\al^\vee\ra}\frac{db_\al}{|b_\al|}.
\end{equation*}

In particular, on the cell $C_0^{\i}(\bm):=C^{\i}(\bm)\cap U^-_{\i}$ such that $m_\al>0$ for each $\al$, if we set $b_\al=\vp^{m_\al}t_\al$, we have
\[
du = q^{\la\rho,\mathrm{wt}(\bm)\ra}\prod_{\al\in\Phi^+}{dt_\al},
\]
where $\mathrm{wt}(\bm) = \sum_{\al\in\Phi^+}m_\al\al^\vee$.

Furthermore, the maps $\eta:U^-_{w_0}\lra U^-_{w_0}$ is unimodular.
\end{Prop}
\begin{proof}
The first equality follows from a comparison to the archimedean computation of \cite[Theorem 5.1.3]{chhaibi2013littelmann} and \cite[Proposition 5.1.6]{chhaibi2013littelmann}, and noting that ultimately we need to compute the Jacobian of the same matrix. Chhaibi shows that this matrix, as well as the Jacobian of $\eta$, has determinant $\pm1$, so that the measures agree and proves the unimodularity of $\eta$. 

In particular, we see that if we restrict to the open dense set such that $u=z_{\i}(a_\bullet)$, then we have
\[
du = \prod_{\al\in\Phi^+}|a_\al|^{-\la \rho,\al^\vee\ra}\frac{da_\al}{|a_\al|}.
\]
By Theorem \ref{Thm: geom algo}, we may conclude the second equality. 
\end{proof}

 Finally, we consider the character value $\psi_\lam(u)$. For each $i\in I$ let $\mathfrak{s}_i :U^-\lra F$ be the function obtained from the composition
\[
U^-\lra U^-/[U^-,U^-] \cong \prod_{j\in I}U_{\al_j}\lra U_{\al_i}\cong F.
\]
Then $\psi_\lam(u) =\prod_{i\in I}\psi\left(\vp^{w_0\lam_i}\mathfrak{s}_i(u)\right)$ and by Proposition \ref{Prop: independence}, $\mathfrak{s}_i(u)=h_i(t_\al,w_\al)$.  
We now show how to compute the diagonal specialization $h_i(t_\al,t_\al)$ this Laurent polynomial by combining Theorem \ref{Thm: geom algo} with properties of generalized minors. For the moment, let $u\in U$ lie in the unipotent radical of $B^+$\footnote{This is simply to align with the conventions of \cite{BZ01}.}
 Let $\chi=\sum_{i\in I}\mathfrak{s}_i$, so that $\psi(u) = \psi(\chi(u))$. 
Consider the unipotent element $v=\overline{w_0}u\overline{w_0}^{-1}\in U^-$. Assume that $v=\zi\in C^{\i}(\bm)$ as in Theorem \ref{Thm: geom algo}. It follows that
\[
u = \overline{w_0}^{-1}\zi\overline{w_0} = \eta_{w_0}(\overline{w_0}y_{\i^{op}}({b}^{op}_\bullet)\overline{w_0}^{-1}) = \eta_{w_0}(x_{\i^{\ast op}}(-{b}_\bullet^{op})),
\]
where the map $\eta_{w_0}: U\lra U$ indicates 
\begin{equation}\label{conjugate}
u\mapsto \overline{w_0}\eta(\overline{w_0}^{-1}u\overline{w_0})\overline{w_0}^{-1}.
\end{equation}

Writing the integral (\ref{eqn: MV integral}) in terms of the toric measure as in Proposition \ref{Prop: measures}, we may remove the signs by a simple change of variables so that we may as well assume that $u =\eta_{w_0}(x_{\i^{\ast op}}({b}_\bullet^{op})).$

\begin{Prop}\label{Prop: character sum}
Fix a dominant cocharacter $\lam\in X_\ast(T)^+$, a long word $\i\in R(w_0)$ and let $(b_\bullet)\in {(F^\times)}^N$. Set $u = \eta_{w_0}(z)$ where $z=x_{\i^{\ast op}}(b_\bullet^{op})$. Then $\chi(\vp^{\lam}u\vp^{-\lam})$  may be expressed as a Laurent polynomial in the variables $(b_\bullet)$ with positive integer coefficients. 
Writing $\chi(\vp^{\lam}u\vp^{-\lam}) = \sum_{i\in I}\vp^{\la\al_i,\lam\ra}\fs_i(u)$, we have
\begin{equation}\label{eqn: character sum} 
\fs_i(u)=\sum_{\pi:\Lam_i\rightarrow w_0s_i\Lam_i}d_\pi \frac{b_1^{c_1(\pi)}\cdots b_N^{c_N(\pi)}}{b_1^{\la \Lam_i,\be^{\i}_i\ra}\cdots b_N^{\la \Lam_i,\be^{\i}_N\ra}},
\end{equation}\label{equation: trails}
where $\be^{\i}_k:= s_{i_1}\cdots s_{i_{k-1}}(\al^\vee_{i_k})$ and the sum ranges over the $\i$-trails from $\Lam_i$ to $w_0s_i\Lam_i$ and $d_\pi\in \zz_{\geq0}$.
\end{Prop}
\begin{proof}
We begin by noting that we may interpret the functions $\fs_i$ in terms of generalized minors: 
\[
\fs_i(u) =\Del_{\Lam_i,s_i\Lam_i}(u).
\]
Recall that for an element $g\in G_0$, we have the Gauss decomposition $g=[g]_+[g]_0[g]_-$. Conjugation by $w_0$ also gives a decomposition of the form $g=[g]_-[g]_0[g]_+$. Writing $u=\eta_{w_0}(z)$ with $z\in U$, we may combine our definition of $\eta$ with (\ref{conjugate}) to see that $u= [\overline{w_0}^{-1}z^T]_+$. Therefore, we have 
\begin{align*}
\Del_{\Lam_i,s_i\Lam_i}(u) &= \Del_{\Lam_i,s_i\Lam_i}([\overline{w_0}^{-1}z^T]_+)\\
					&=\left([[\overline{w_0}^{-1}z^T]_+\overline{s_i}]_0\right)^{\Lam_i}\\
					&=\left([[\overline{w_0}^{-1}z^T]_0^{-1}[\overline{w_0}^{-1}z^T]_{0+}\overline{s_i}]_0\right)^{\Lam_i}\\
					&= \frac{\left([\overline{w_0}^{-1}z^T\overline{s_i}]_0\right)^{\Lam_i}}{\left([\overline{w_0}^{-1}z^T]_0\right)^{\Lam_i}}\\
					&=\frac{\Del_{s_i\Lam_i,w_0\Lam_i}(z)}{\Del_{\Lam_i,w_0\Lam_i}(z)}.
\end{align*}
Recalling that we have assumed we may write $z=x_{\i^{\ast op}}(b_\bullet^{op})$, the properties of generalized minors (\ref{eqn: involutions}) imply
\[
\Del_{\ga,\de}(z) = \Del_{w_0\de,w_0\ga}(x_{\i}(b_\bullet)).
\] 
Applying Theorem \ref{Thm 5.8} to the minor ${\Del_{s_i\Lam_i,w_0\Lam_i}(z)}$, we have
\begin{equation*}
\Del_{s_i\Lam_i,w_0\Lam_i}(z)=\sum_{\pi:\Lam_i\rightarrow w_0s_i\Lam_i}d_\pi b_1^{c_1(\pi)}\cdots b_N^{c_N(\pi)},
\end{equation*}
where $d_\pi\in\zz_{>0}$.
To conclude we note that $\Del_{\Lam_i,w_0\Lam_i}(z) = \Del_{\Lam_i,w_0\Lam_i}(x_{\i}(b_\bullet))$.  Corollary 9.5 of \cite{BZ01} implies that there is a \emph{unique} $\i$-trail from $\Lam_i$ to $w_0\Lam_i$, so that the denominator is in fact a monomial. More specifically, setting $\be^{\i}_k:= s_{i_1}\cdots s_{i_{k-1}}(\al^\vee_{i_k})$, we have
\begin{equation*}\label{unique trail}
\Del_{\Lam_i,w_0\Lam_i}(x_{\i}(b_\bullet))=\prod_{k=1}^Nb_k^{\la \Lam_i,\be^{\i}_k\ra}.
\end{equation*}\qedhere
\end{proof}

Applying \cite[Theorem 7.8]{McN}, this formula tells us when a cell $C^{\i}(\bm)$ corresponds to an MV cycle in the finite crystal graph $\B(\lam)$ for some dominant coweight $\lam$. 
\begin{Thm}\label{Thm: finite crystal}
Let $\lam=\sum_\al\lam_\al\Lam^\vee_\al\in X_\ast(T)^+$ be a dominant coroot, and let $\B(\lam)$ be the finite highest-weight crystal associated to the complex dual Lie algebra $\fg^\vee$. Then $(\i,\bm)\in \B(\lam)$ if and only if for each $\al\in \Del$ and for each $\i$-trail $\pi:\Lam_\al\rightarrow w_0s_\al\Lam_\al$, the following inequality holds:
\[
\sum_{k=1}^N\la \Lam_\al,\be^{\i}_k\ra m_k\leq \lam_\al +\sum_{k=1}^Nc_k(\pi)m_k.
\]
Here $c_i(\pi)$ are the coefficients associated to the  $\i$-trail $\pi$ as in the definition.
\end{Thm}
\begin{Rem}
The set of $\i$-trails here giving the upper bounds of the finite crystal graph are precisely the same as those shown in \cite[Theorem 3.10]{BZ01} to parametrize the $\i$-string cone, reflecting the duality between the two parametrizations of canonical bases.
\end{Rem}

\subsection{A monomial change of variables}\label{Section: boundary}
We noted in Proposition \ref{Prop: independence} that the Laurent polynomial $\fs_i(u)=h_i(t_\al,w_\al)$ is independent of $\bm$. The relations between the variables $t_k$ and $w_k$ on the other hand do depend on the $\i$-Lusztig data. We now discuss a method of recovering a Laurent polynomial $g_i(\underline{x},\underline{y})$ such that under the diagonal specialization $g_i(\underline{x},\underline{x})=h_i(\underline{x},\underline{x})$ which in a precise way recovers as much of $h_i(\underline{x},\underline{y})$ as possible. In fact, $g_i(\underline{x},\underline{y})=h_i(\underline{x},\underline{y})$ for certain reduced words, including the long word for $G_2$ we consider in Section \ref{Section: G2 integrals}.

Let $\ga_k\in \Phi^+$, and suppose $\al_i$ is the unique simple root such that $\al_i<\ga_k$, and no simple root lies between these roots. Choose a monomial $X_k$ of $\mathfrak{s}_i(u)$ of the form
\[
X_k = \frac{b_1^{c_1(\pi)}\cdots b_N^{c_N(\pi)}}{b_1^{\la \Lam_\al,\be^{\i}_1\ra}\cdots b_N^{\la \Lam_\al,\be^{\i}_N\ra}},
\]
where
\begin{equation}\label{desiderata 1}
c_j(\pi) = {\la \Lam_\al,\be^{\i}_j\ra}\:\:\text{ for each}\:\: j<k,\: \text{ and }\:\: c_k(\pi)={\la \Lam_\al,\be^{\i}_k\ra}-1.
\end{equation}
In other words, if we express $X_k$ in reduced form, then $b_k$ is the lowest index variable in the denominator of $X_k$, and its exponent is $1$. 
This choice is made so that if $Y_k$ is the monomial of $h_i(t_\al,w_\al)\in \zz[t_\al,w_\al^{\pm1}]$ specializing to $X_\ga$, then $Y_k$ be of degree $1$ in the $t_\al$-variables with only $t_{k}$ appearing; that is,
\begin{equation}\label{desiderata 2}
Y_k= t_{k}\prod_{j}w_j^{m(k,j)}
\end{equation}
for some $m(k,j)\in \zz$. Combining this with Theorem \ref{Thm: geom algo} and (\ref{desiderata 1}), we may assume $m(k,j)=0$ for $j\leq k$. 
It is easy to show that such a monomial exists; in general, it is not unique.

 By comparing (\ref{desiderata 2}) to Proposition \ref{Prop: character sum}, it is straightforward to determine the corresponding $\i$-trail from the exponents. Fix such a choice of monomials $\{X_k\}_{i=1}^N$ (for example, by ordering $\i$-trails lexicographically, we may choose the greatest possible element as our $X_k$). This associates to each positive root $\ga_k\in \Phi^+$ an $\i$-trail $\pi$ and induces a monomial change of variables
 \[
 \mathfrak{s}_i(u)=\sum_{\pi:\Lam_i\rightarrow w_0s_i\Lam_i}d_\pi X_{1}^{c_\pi(1)}\cdots X_{N}^{c_\pi(N)},
 \]
 for some $c_\pi(k)\in \zz$ such that for each $\pi$
\[
\sum_{k=1}^Nc_\pi(k)=1,
\]
since the total degrees of the monomials is $-1$ in the $b_\bullet$-variables. Define $s_k =\lam_{i^\ast}+\val(X_k)$, so that $X_k=\vp^{s_k-\lam_{i^\ast}}W_k.$ Applying Proposition \ref{Prop: measures}, for $\bm$ such that $m_k>0$ for all $k$,  
 \begin{equation}\label{eqn: general formula}
 I_\lam(\bm)= \displaystyle\prod_{\al}x_{\al}^{m_\al}\int_{(\calo^\times)^N}\prod_{k}(W_k,\vp)^{r_k}\prod_{i\in\Del}\psi_\lam\left(\sum_{\pi:\Lam_i\rightarrow w_0s_i\Lam_i}d_\pi\vp^{n(\pi)} W_{1}^{c_\pi(1)}\cdots W_{N}^{c_\pi(N)}\right)\prod_{k}dW_k,
 \end{equation}
 for uniquely determined $r_k, n(\pi)\in \zz$. Note that $(\i,bm)\in \B(\lam+\rho)$ if and only if $n(\pi)\geq-1$ for all $\i$-trails. As we shall see in the next section, unlike the type $A$ case in \cite{McN}, it is not true that this may be written as a product over the positive roots without further analysis. For any given case, this formula allows one to show that $I_\lam(\bm)=0$ for \emph{generic} $\bm\notin \B(\lam+\rho)$. It is non-trivial to determine the precise locus of non-vanishing; see Section \ref{Section: resonance} below.

Finally, we define
 \[
 g_i(t_\al,w_\al)=\sum_{\pi:\Lam_i\rightarrow w_0s_i\Lam_i}d_\pi Y_{1}^{c_\pi(1)}\cdots Y_{N}^{c_\pi(N)}\prod_k\left(\frac{t_k}{w_k}\right)^{-\min\{0,c_\pi(k)\}}.
 \]
 By construction, it is clear that $g_i(t_\al,t_\al)=h_i(t_\al,w_\al)$. 
 In general, it is not true that $g_i(t_\al,w_\al) =h_i(t_\al,t_\al)$: there may be several distinct monomials in $h_i(t_\al,w_\al)$ which specialize to a given monomial $X_k$. This occurs with the other word $\i=(1,2,1,2,1,2)$ in type $G_2$. Additionally, it can happen that there are pairs of monomials in $h_i(t_\al,w_\al)$ which cancel each other out upon specialization. This happens in type $A_3$ with respect to the long word $\i=(1,2,3,2,1,2)$. Nevertheless, we expect that the $I_\lam(\bm)$ can always be computed in terms of $g_i(t_\al,w_\al)$.
\begin{Conj}\label{Conjecture}
Fix the root system $\Phi$ and the long word $\i$. Let $\bm\in \B(-\infty)$ be an $\i$ -Lusztig datum, and let $\lam$ be a dominant coweight. Then
\begin{equation}\label{conj integral}
I_\lam(\bm)=  \int_{C^{\i}(\bm)}f(u)\psi\left(\sum_{i\in I}\vp^{\lam_{i^\ast}}g_i(t_\al,w_\al)\right)du.
\end{equation}
In other words, we may replace the sum over the Laurent polynomials $h_i(t_\al,w_\al)$ with the sum over the polynomials $g_i(t_\al,w_\al)$ so that the natural generalization of the general formula (\ref{eqn: general formula}) holds for all $\i$-Lusztig data.
\end{Conj}

In Appendix \ref{App: type A}, we verify the conjecture for the only word of type $A_3$ for which $g_i(t_\al,w_\al)\neq h_i(t_\al,w_\al)$. For certain words $\i$, this conjecture is vacuously true as one may argue in terms on the properties of the induced ordering of the positive roots to see that indeed $g_i(t_\al,w_\al)=h_i(t_\al,w_\al)$.
 
 \begin{Prop}\label{Prop: some words}
 For the Gelfand-Tsetlin word $\i=(1,2,\ldots,r,1,2,\ldots,1,2,1)$  in type $A_r$, both long words of type $B_2$, and the long word $\i=(2,1,2,1,2,1)$ for $G_2$ where $\al_1$ (resp. $\al_2$) is the long (short) simple root, we have that $g_i(t_\al,w_\al)=h_i(t_\al,w_\al)$ for all $i\in I$.
 \end{Prop}

 Assuming Conjecture \ref{Conjecture} for a given long word $\i$, we may now compute the associated MV integrals. This requires determining the relevant set of $\i$-trails, which is already formidable in general. In the following section, we consider the case of $G_2$ and $\i=(2,1,2,1,2,1)$ in detail. The analysis isolates certain families of Lusztig data \emph{outside} of the finite crystal $\B(\lam+\rho)$ which contribute to the final sum: these correspond to resonant MV polytopes. We give a general definition and discussion on this notion in Section \ref{Section: resonance}. 

\section{MV integrals for $G_2$}\label{Section: G2 integrals}
In this section, we apply the results of the previous sections to the case of the exceptional group $G_2$ with the long word $\i=(2,1,2,1,2,1)$, where $\al_1$ is the long simple root and $\al_2$ is the short simple root. For a fixed dominant coweight $\lam$, we then compute the MV integrals, finding that there are infinitely many \emph{resonant} Lusztig data outside the finite crystal $\B(\lam+\rho)$ with non-zero contribution. For simplicity, we only state results for the non-covering ($n=1$) case; for higher degree covers, the integrals are computable but the exponential sums that appear depend on the parity of the cover and on other divisibility constraints.  In the next section, we show how a new crystal graph structure on these resonant MV polytopes allows us to derive a Tokuyama-type formula for $G_2$.
\newpage
\begin{center}
\includegraphics[scale=0.35]{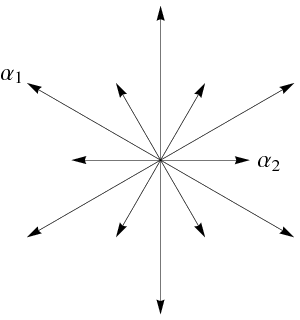}
 \captionof{figure}{$\Phi=G_2$}
\end{center}
\subsection{Preliminaries}
For our choice of $\i$, the corresponding convex ordering of the positive roots $\{\ga_i\}$ is 
\[
\al_2<\al_1+3\al_2<\al_1+2\al_2<2\al_1+3\al_2<\al_1+\al_2<\al_1.
\]

 Setting $u=\eta_{w_0}(x_{\i^{\ast op}}(b_\bullet^{op}))$ as in the proof of Proposition \ref{Prop: character sum}, the necessary generalized minors are 
\[
\Del_{\Lam_k,w_0s_k\Lam_k}(x_{\i_M}(b_\bullet)), \quad \text{and}\quad \Del_{\Lam_k,w_0\Lam_k}(x_{\i_M}(b_\bullet)),
\]
for $k=1$ and $2$. For $k=1$, there is a single $\i$-trail, so that
\begin{equation*}
\Del_{\Lam_1,w_0s_1\Lam_1}(x_{\i_M}(b_\bullet))= b_2b_3^3b_4^2b_5^3,\quad\Del_{\Lam_1,w_0\Lam_1}(x_{\i_M}(b_\bullet))=b_2b_3^3b_4^2b_5^3b_6.
\end{equation*}
Thus, 
\begin{equation*}
\Del_{\Lam_1,s_1\Lam_1}(u) =\frac{b_2b_3^3b_4^2b_5^3}{b_2b_3^3b_4^2b_5^3b_6}=\frac{1}{b_6},
\end{equation*}
which also illustrates the general inductive nature of these computations \cite[Prop. 9.4]{BZ97}. For $k=2$, there are $6$ $\i$-trails from $\Lam_2$ to $w_0s_2\Lam_2$, so that 

\begin{eqnarray*}
\Del_{\Lam_1,w_0s_1\Lam_1}(x_{\i_M}(b_\bullet))=b_1b_2b_3^2b_4+ b_1b_2b_5^2b_6+ b_1b_4b_5^2b_6
									        + 2b_1b_2b_3b_5b_6+ b_1b_2b_3^2b_6+b_3b_4b_5^2b_6.	
\end{eqnarray*}

In general, there is no known formula for the positive integral coefficients arising in generalized minors. For the purposes of $p$-adic integrals, this is not a central issue provided the residue characteristic of $F$ is sufficiently high. In any case, for $G_2$ a simple matrix calculation allows us to find the coefficients above.

Finally, $\Del_{\Lam_2,w_0\Lam_2}(x_{\i_M}(b_\bullet))=b_1b_2b_3^2b_4b_5$, so that
\begin{eqnarray*}
\Del_{\Lam_2,s_2\Lam_2}(u)&= \displaystyle\frac{1}{b_5}+\frac{b_5b_6}{b_3^2b_4}+\frac{b_5b_6}{b_2b_3^2}
					+ 2\frac{b_6}{b_3b_4}+\frac{b_5b_6}{b_1b_2b_3}+\frac{b_6}{b_4b_5}.
\end{eqnarray*}

The monomial change of variables is unique in this case, so we set
\begin{align*}
& X_1= \frac{b_5b_6}{b_1b_2b_3}, \quad X_4=\frac{b_6}{b_4b_5},
\quad X_2 = \frac{b_5b_6}{b_2b_3^2},\\ 
& X_5=\frac{1}{b_5},t
\quad X_3 = \frac{b_6}{b_3b_4}, \quad X_6=\frac{1}{b_6}.\\
\end{align*}
With this change of variables, we rewrite our character entry in terms of the $X_i$:
\begin{align*}
\Del_{\Lam_2,s_2\Lam_2}(u)=X_1&+X_2+2X_3+\frac{X_3^2}{X_4}+X_4+X_5,\\
&\Del_{\Lam_1,s_1\Lam_1}(u)=X_6
\end{align*}

As remarked in the previous section, we may recover the complete Laurent polynomials $h_i(t_\al,w_\al)$ in this case. In particular, we find that $h_1(t_\al,w_\al)=t_6$ and
\[
h_2(t_\al,w_\al)=\displaystyle\frac{t_1w_2w_3}{w_5w_6}+\frac{t_2w^2_3}{w_5w_6}+t_5+2\frac{t_3w_4}{w_6}+\frac{t_4w_5}{w_6}+\frac{t_3^2w_4}{w_5w_6}.
\]

Taking valuations of the monomials, we have the bounding data $s_\al=\val(\vp^{\lam_i}X_\al)$:
\begin{align*}
s_1=\lam_2+m_5+m_6-m_1-m_2-m_3, &\quad s_2=\lam_2+m_5+m_6-m_2-2m_3,\\
s_3=\lam_2+m_6-m_3-m_4,&\quad s_4=\lam_2+m_6-m_4-m_5,\\
s_5=\lam_2-m_5,&\quad s_6=\lam_1-m_6,\\
2s_3-s_4=\lam_2+m_5&+m_6-m_4-2m_3.
\end{align*}

 By Theorem \ref{Thm: finite crystal}, $(\i,\bm)\in \B(\lam+\rho)$ if and only if $s_\al\geq -1$ for each $\al\in \Phi^+$ along with the additional requirement that $2s_3 -s_4\geq -1$.
Given a $\i$-Lusztig datum $\bm$, we say that the positive root $\ga_k$ is \emph{circled} if $m_k=0$. This is motivated by the combinatorial decorations of \cite{BBF} in type A, and refers to those elements of the finite crystal $\B(\lam+\rho)$ which are minimal in sense that they lie on a lower bounding hyperplane for the crystal. Dually, we say that the positive root $\ga_k$ is \emph{boxed} when the corresponding element of the finite crystal lies on a bounding hyperplane coming from the restriction to the highest weight $\lam+\rho$: this means $s_k=-1$. We remark that the terminology of circling and boxing patterns is not sufficient for the additional restriction $2s_3-s_4=-1$.

Finally, we will see below that those Lusztig data for which  
\[
\val(X_3)=\val(X_4)=\val\left(\frac{X_3^2}{X_4}\right),
\]
will play an important role in the evaluation outside $\B(\lam+\rho)$; this is equivalent to the condition $m_3=m_5$, and we say that such a Lusztig datum (or the corresponding MV polytope $M$) is \textbf{resonant} with respect to $\i$ (or simply resonant). For reasons that will become clear in the next section, we say that $M$ is \textbf{totally resonant} if it has the above resonance and $m_2=m_6$ as well.
 \begin{figure}
\centering
\begin{minipage}{.5\textwidth}
  \centering
  \includegraphics[width=.6\linewidth]{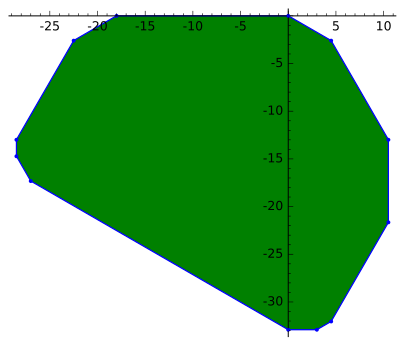}
  \caption{Resonant MV polytope}
  \label{fig:test1}
\end{minipage}%
\begin{minipage}{.5\textwidth}
  \centering
  \includegraphics[width=.6\linewidth]{mv_G1}
  \caption{Totally resonant MV polytope}
  \label{fig:test2}
\end{minipage}
\end{figure}

\subsection{Analysis of the MV integrals for $n=1$}

Our primary motivation is to find generalized Tokuyama-type formulas and study the relation to canonical bases. For this reason, we shall only focus on the $n=1$ case and consider the integral $I_\lam(\bm)$ given by
\begin{equation*}
\prod_{\al\in \Del^+}(q^{-1}x_\al)^{m_\al}\int_{C^{\i}(\bm)} \psi\left(\vp^{\lam_1}t_6+\vp^{\lam_2}\left(\frac{t_1w_2w_3}{w_5w_6}+\frac{t_2w^2_3}{w_5w_6}+t_5+2\frac{t_3w_4}{w_6}+\frac{t_4w_5}{w_6}+\frac{t_3^2w_4}{w_5w_6}\right)\right)dt.
\end{equation*}
We must consider cases, depending on certain resonances; the results are summarized in Theorem \ref{Thm: all together}. Before proceeding, we introduce some more notation. Define
\begin{equation*}
I(a,b) = \begin{cases} q^{a+b}\int_{\vp^a\calo}\psi(t)dt \qquad\;\:: b=0,\\
					q^{a+b}\int_{\vp^a\calo^\times}\psi(t)dt\qquad: b>0.\end{cases} 
\end{equation*}
Note that $I(a,b) = 0$ if $a<-1$ or if $a=-1$ and $b=0$. This simple integral controls much of the behavior of $I_\lam(\bm)$. For example,  the first vanishing statement above often tells us that $I_\lam(\bm)=0$ if $(\i,\bm)\notin \B(\lam+\rho)$. The second property corresponds to the vanishing of terms associated to non-strict Gelfand-Tsetlin patterns.

We begin with the change of variables
\[
X_1=\vp^{\lam_2}\frac{t_1w_2w_3}{w_5w_6},\quad\text{and}\quad X_2 = \vp^{\lam_2}\frac{t_2w_3^2}{w_5w_6},
\]
so that
\begin{equation*}\label{eqn: factor out 1,2}
I_\lam(\bm) = I(s_1,m_1)I(s_2,m_2) J_\lam(\bm)\prod_{\al}(q^{-1}x_{\al})^{m_\al},
\end{equation*}
where
\begin{equation*}\label{eqn: J integral}
J_\lam(\bm)= \int\int\int\int \psi\left(\vp^{\lam_1}t_6+\vp^{\lam_2}\left(t_5+2\frac{t_3w_4}{w_6}+\frac{t_4w_5}{w_6}+\frac{t_3^2w_4}{w_5w_6}\right)\right)dt.
\end{equation*}
This latter integral controls the behavior of the MV integrals as we vary over circling patterns. 
 We recall the Tokuyama function $G(\bm)$\footnote{Here we ignore the dependence on $\i$.}, where 
\begin{equation*}
G(\bm) = \prod_{\al\in \Del^+}G(s_\al,m_\al)
\end{equation*}
where 
\begin{equation*}
G(s_\al,m_\al)=\begin{cases} \qquad\quad1-q^{-1} \qquad\;: m_\al>0, s_\al\geq 0\\ \qquad\qquad-q^{-1}\qquad\;: m_\al>0, s_\al=-1,\\\qquad\qquad 1\qquad\qquad: m_\al=0, s_\al\geq 0,\\\qquad\qquad 0 \qquad\qquad:\text{otherwise}.\end{cases}
\end{equation*}
Fix the notation $x^{k(\bm)}=x_1^{k_1}x_2^{k_2}:=\prod_\al x_{\al}^{m_\al}$. 
In particular, if $2s_3-s_4\geq0$ we have that 
\begin{equation}\label{standard contribution}
I_\lam(\bm) = G(\bm)x_1^{k_1}x_2^{k_2}.
\end{equation}
 This is referred to as the standard contribution, as it describes the contributions arising in the classical Gelfand-Tsetlin case in type A \cite{Tok,BBF,McN}. 
There are many cases when $2s_3-s_4<0$ that (\ref{standard contribution}) holds. A useful observation is that certain circling patterns force this to be the case.
\begin{Lem}\label{Lem: simplifying}
If $m_4=0$, then $I_\lam(\bm)$ matches the standard contribution.
\end{Lem}
\begin{proof}
As noted above, if $2s_3-s_4\geq 0$, this is already known. Now suppose that $m_4=0$ and $2s_3-s_4<0$. Then we observe that $$2s_3-s_4=s_2+m_2<0.$$ Since (\ref{eqn: factor out 1,2}) shows that $I(s_2,m_2)$ is a factor of $I_\lam(\bm)$ for an arbitrary $(\i,\bm)\in\B(\infty)$, we know that $s_2< -1$ forces $I_\lam(\bm)=0,$ which matches the standard contribution in this case. 

Otherwise, $2s_3-s_4=s_2+m_2=-1,$ which forces $m_2=0$ and $s_2=-1$. But then we have $$I(s_2,m_2) = I(-1,0)=0,$$ so that $I_\lam(\bm)=0$ matches the standard contribution again. This exhausts the cases.
\end{proof}
We thus need only consider whether $\ga_3$ is circled or not, while imposing the assumption that $m_4>0$. 

\noindent
\textbf{Case $1$:} \underline{$m_3>0$} By a simple change of variables, we have that 
\[
J_\lam(\bm)=I(s_5,m_5)I(s_6,m_6)I_\lam(s_3,s_4;m_3,m_4),
\] 
where 
\begin{equation*}
I_\lam(s_3,s_4;m_3,m_4):=q^{m_3}q^{m_4}\int_{\calo^\times}\int_{\calo^\times}\psi\left(2\vp^{s_3}t_3+\vp^{s_4}t_4+\vp^{2s_3-s_4}\frac{t_3^2}{t_4}\right)dt_3d_4.
\end{equation*}
 If $2s_3\geq s_4$, then this integral reduces to 
\begin{equation*}\label{case 0}
I_\lam(s_3,s_4;m_3,m_4)=I(s_3,m_3)I(s_4,m_4).
\end{equation*}
 In particular, $I_\lam(\bm)=0$ for $(\i,\bm)\notin \B(\lam+\rho)$. If we have $2s_3<s_4$ and $0\leq s_4$, then

after a simple change of variables
\begin{equation*}\label{case 1}
I_\lam(s_3,s_4;m_3,m_4)= I(s_3,m_3)I(2s_3-s_4,m_4).
\end{equation*}

 If we have $2s_3<s_4<0$, then all three terms in the additive character contribute.  Let us make the change of variables $x=\vp^{s_3}t_3$ and $y=\vp^{s_4}t_4$ so that 
\begin{equation*}\label{eqn: archtype}
I_\lam(s_3,s_4;m_3,m_4) = q^{s_3+m_3}q^{s_4+m_4}\int_{\vp^{s_3}\calo^\times}\int_{\vp^{s_4}\calo^\times}\psi\left(2x+\frac{x^2}{y}+y\right)dydx.
\end{equation*}
We examine two cases:

\noindent
$\textbf{Case $1.1$:  } \underline{s_3\neq s_4}$. Note that none of the terms in this case correspond to $\i$-Lusztig data in $\B(\lam+\rho)$, since we cannot have $s_3=s_4=-1$. 

We begin by making the change of variables $z=y/x$, so that
\begin{equation*}
I(x,y)= q^{-s_3}\int_{\vp^{s_4-s_3}\calo^\times}\int_{\vp^{s_3}\calo^\times}\psi\left(x\frac{(1+z)^2}{z}\right)dxdz.
\end{equation*}
Making a change of variables $x \mapsto x\frac{(1+y)^2}{y}$ in the inner integral, we have
\begin{eqnarray*}
I_\lam(s_3,s_4;m_3,m_4)&=\begin{cases} \;\;I(-s_3,m_3)I(s_4,m_4) \qquad\qquad: s_4<s_3\\
				 I(-s_3,m_3)I(2s_3-s_4,m_4)\qquad: s_4>s_3.\end{cases}
\end{eqnarray*}
In both cases with non-vanishing integral, we obtain the contradiction of $-1<s_3<0$: {in the first case, we need $s_4=-1<s_3<0$ and in the second case, $2s_3-s_4=-1<s_3<0$.} In particular, $I_\lam(s_3,m_3,s_4,m_4)=0$ in all of these cases.\\

\noindent
 $\textbf{Case $1.2$:  } \underline{s_3=s_4}.$ Note that in this case $(\i,\lam)\in \B(\lam+\rho)$ if and only if $k:=s_3=s_4=-1$. We also note that this is precisely the resonant case $m_3=m_5>0$.  
By changing variables as in the previous case,
\begin{equation*}\label{case 2}
I_\lam(s_3,s_4;m_3,m_4)= q^{s_3+m_3}q^{s_4+m_4}\left[q^{-s_3}\int_{\calo^\times}\int_{\vp^{s_3}\calo^\times}\psi\left(x\frac{(1+y)^2}{y}\right)dxdy\right].
\end{equation*}
We now decompose this over $\calo^\times =\bigsqcup_{\zeta\in \calo^{\times}/\la\vp\ra}\tilde{\zeta}+\vp\calo$ to obtain
\begin{align*}
q^{-k-1}\sum_{\zeta\neq 0,-1} \int_\calo\int_{\vp^k\calo^\times}\psi\left(x\frac{(1+\tilde{\zeta}+\vp y)^2}{\tilde{\zeta}+\vp y}\right)dxdy\\ \qquad+q^{-k-1}\int_\calo\int_{\vp^k\calo^\times}\psi\left(x\frac{(\vp y)^2}{-1+\vp y}\right)dxdy.
\end{align*}
For each $\zeta\neq-1$, this may be evaluated in a similar fashion to the previous case as $\frac{(1+\tilde{\zeta}+\vp y)^2}{\tilde{\zeta}+\vp y}$ has constant valuation $0$. Thus, the first line is equal to
\[
q^{-k}(1-2q^{-1})(q^{-k})\int_{\calo^\times}\psi(\vp^kx)dx = \begin{cases}\quad-q(1-2q^{-1})\qquad: k=-1\\\qquad \qquad0\qquad\quad\quad\;: k<-1.\end{cases}
\]
In the case $\zeta=-1$,
\begin{align*}
q^{-k-1}\int_\calo\int_{\vp^k\calo^\times}\psi\left(x\frac{(\vp y)^2}{-1+\vp y}\right)dxdy 
&=q^{-k+1}(1-q^{-1})\sum_{j=0}^\infty q^{j}\int_{\vp^{k+2j+2}\calo^\times}\psi(x)dx.
\end{align*}
Thus, we have in the case that $s_3=s_4=k<0$
\begin{equation*}
I_\lam(k,k;m_3,m_4) = \begin{cases} \qquad I(s_3,m_3)I(s_4,m_4)\qquad\;\;\;: k=-1\\\qquad q^{m_3}q^{m_4}q^{-n}(1-q^{-1})\qquad: k=-2n\leq-2\\\qquad \qquad0\qquad\quad\qquad\qquad\quad\;: \text{otherwise}.\end{cases} 
\end{equation*}
 We summarize the results of this case here:
\begin{equation*}\label{case1}
I_\lam(\bm) = \prod_{\al\neq \ga_3,\ga_4}I(s_\al,m_\al)\prod_{\al}(q^{-1}x_{\al})^{m_\al}\begin{cases} I(s_3,m_3)I(s_4,m_4)\qquad\qquad: 	2s_3\geq s_4\\
																								I(s_3,m_3)I(2s_3-s_4,m_4)\quad\;\;: 2s_3<s_4, 0\leq s_4\\
																								\quad q^{m_3}q^{m_4}q^{-n}(1-q^{-1})\qquad\;: s_3=s_4=-2n\leq-2\\
																								\qquad\qquad 0\qquad\qquad\qquad\qquad: \text{otherwise}\end{cases}
\end{equation*}
\begin{Rem}
This second case is new, as it can truly differ from the standard contribution inside $\B(\lam+\rho)$.
Moreover, there are cases where $I_\lam(\bm)=0$ despite a non-trivial standard contribution.  
In particular, the standard contribution requires augmenting to encode the subtle geometry of general highest-weight crystal graphs.
\end{Rem}

\noindent
 \textbf{Case $2$:  } \underline{$m_3=0$}. In this case, we factor off $I(s_5,m_5)$ and consider the inner integration
  \begin{equation*}
  I=\int_{\vp^{-m_4}\calo^\times}\int_\calo\psi_\lam\left(2\frac{t_3 t_4}{w_6}+\frac{t_3^2t_4}{w_5w_6}+\frac{t_4w_5}{w_6}\right)dt_3dt_4.
 \end{equation*}
Setting $$x= \vp^{\lam_2}\frac{t_4w_5}{w_6} \text{     and    }y = \frac{t_3}{w_5},$$ this integral reduces to 
\begin{equation*}
\displaystyle q^{m_5}q^{s_4+m_4}\int_{\vp^{s_4}\calo^\times}\int_{\la\vp^{m_5}\ra}\psi_\lam\left(x(y+1)^2\right)dydx.
\end{equation*}
\textbf{Case $2.1$:} If $m_5>m_3=0$, 
 a simple change of variables this integral gives
$
I(s_4,m_4),
$
 so that if $m_3=0$ and $m_4,m_5>0$, we find
 \begin{equation*}
 J_\lam(\bm) = I(s_3,m_3)I(s_4,m_4)I(s_5,m_5)I(s_6,m_6).
 \end{equation*}
We remark that since this requires $s_4\geq-1$ to not vanish, we must have that $s_3>0$, allowing us to insert $I(s_3,m_3) = 1$ so that this formula agrees with (\ref{standard contribution}).  Note again that only  MV cycles in $\B(\lam+\rho)$ contribute in this case.

\noindent
\textbf{Case $2.2$:} In the resonant case $m_5=m_3=0$, we have $s_3=s_4=2s_3-s_4$. The integral must now be decomposed so that it may be written as
\begin{align*}
q^{s_4+m_4}\left[\sum_{\zeta\neq0}q^{-1}\int_{\vp^{s_4}\calo^\times}\int_{\calo}\psi\left(x(\vp t+\tilde{\zeta})^2\right)dtdx+ q^{-1}\int_{\vp^{s_4}\calo^\times}\int_{\calo}\psi\left(x(\vp t)^2\right)dtdx\right].
\end{align*}

Considering the first term, since $\val(x(\vp t+\tilde{\zeta})^2)=s_4$, this simplifies after a change of variables to
\[
(1-q^{-1})I(s_4,m_4).
\]
For $2s_3-s_4<-1$, this term vanishes. As for the second term, we may rewrite this as
\begin{align*}
q^{s_4+m_4}\left[q^{-s_4-1}(1-q^{-1})\sum_{j=0}^\infty q^{-j}\int_{\calo^\times}\psi\left(\vp^{s_4+2j+2}x\right)dx\right].
\end{align*}
This now evaluates similarly to Case 1.2 above, and we obtain
\begin{equation*}
 I=\begin{cases}  \qquad I(s_4,m_4)\qquad\quad\;\;\: : s_4\geq 0\\
													\qquad \quad\:0\qquad\qquad\qquad: s_4\leq 0 \text{ odd}\\
														q^{-n}\left[q^{m_4}(1-q^{-1})\right]\qquad: s_3=s_4=-2n< 0, \end{cases}
\end{equation*}
Therefore, we find that 
\begin{equation*}
J_\lam(\bm) = \begin{cases} I(s_3,m_3)I(s_4,m_4)I(s_5,m_5)I(s_6,m_6)\quad: m_3=0,\; s_4\geq -1\\  I(s_5,m_5)I(s_6,m_6)q^{-n}\left[q^{m_4}(1-q^{-1})\right]\ \quad :m_3=m_5=0, s_4=-2n<0.\end{cases}
\end{equation*}
\subsection{Putting it all together}
Fix a dominant weight $\lam$, and let $\bm\in \B(-\infty)$ be a Lusztig datum with respect to the long word $\i=(2,1,2,1,2,1)$. The computations of the preceding section prove the following
\begin{Thm}\label{Thm: all together}
Let $I_\lam(\bm)$ be the MV integral associated to $(\i,\bm)$. Set $k(\bm)=(k_1,k_2)$. Then if $\bm\in \B(\lam+\rho)$ 
\begin{align*}\label{eqn: evaluation}
I_\lam(\bm)
		&= \prod_{\al\neq\ga_4}G(s_\al,m_\al)G(2\min\{s_3,s_4\}-s_4,m_4)x_1^{k_1}x_2^{k_2}=:G_1(\bm)x_1^{k_1}x_2^{k_2}.
\end{align*}
If $\bm\notin \B(\lam+\rho)$ then $I_\lam(\bm) =0$ unless $\bm$ is resonant ($m_3=m_5$), $m_4>0$, and $s_3=s_4=-2k<0$. In this case, we find
\begin{equation*}\label{eqn: outside evaluation}
I_\lam(\bm)=q^{-k}(1-q^{-1})\prod_{\al\neq \ga_3,\ga_4}G(s_\al,m_\al)x_1^{k_1}x_2^{k_2}=:G_{res}(\bm)x_1^{k_1}x_2^{k_2}.
\end{equation*}
\end{Thm}

 \section{Resonance families}\label{Section: resonance families}
Theorem \ref{Thm: all together} shows that there are infinitely many Lusztig data $\bm$ such that $I_\lam(\bm)\neq0$. In this section, we reduce to a sum over a finite set of Lusztig data by taking advantage of an additional crystal graph structure on resonant Lusztig data outside $\B(\lam+\rho)$. We find that despite being able to reduce to a sum over finitely many Lusztig data, we have
\[
\sum_{\bm\notin \B(\lam+\rho)}I_\lam(\bm)\neq 0,
\]
so that a straightforward Tokuyama-type formula does not hold in this case. Furthermore, while it is possible to incorporate those non-trivial terms from outside the highest-weight crystal into a sum over $\B(\lam+\rho)$, there does not appear to be a canonical way to do so. For this reason, the formula we give in Theorem \ref{Thm: tokuyamatype} is not given simply as a sum over $\B(\lam+\rho)$, but has a correction factor arising from certain families of such Lusztig data.
\subsection{Arrays}\label{section: arrays}
To define these families, we introduce some additional notation. Define
\[
\B(\lam+\rho)_{res}=\left\{\bm : m_3=m_5;\; s_i\geq-1 \text{ for }i\neq3,4;\; s_3=s_4\leq0\text{ even}\right\}.
\]
In particular, if $\bm\notin \B(\lam+\rho)$ is a Lusztig datum such that $I_\lam(\bm)\neq0$, then $\bm\in \B(\lam+\rho)_{res}$. 
We introduce the following combinatorial notation: 
\begin{Def}
We call $A$ an \textbf{array of weight}{ $\kappa$} if 
\[
A=\left[\begin{array}{ccccc} a&b&c&b&d\\ &x&y&z&\\&&k&&\end{array}\right],
\] with all entries are non-negative integers and where 
\[
\kappa=a+6b+2c+d.
\]
We refer to the set of all arrays as $Arr$, and the set of weight $\kappa$ by $Arr_{\kappa}$. The value $k\geq0$ in the bottom row is the \textbf{decoration} of the array.
\end{Def}

For any $\bm\in \B(\lam+\rho)_{res}$ we have the \emph{array associated to} $\bm$:
\[
A(\bm) = \left[\begin{array}{ccccc} m_2&m_5&m_4&m_5&m_6\\ &s_2+1&s_6+1&s_5+1&\\&&k&&\end{array}\right],
\]
where $s_3=s_4=-2k$. Note that the weight of $A(\bm)$ is $k_1$, where $k(\bm) = (k_1,k_2)$. This gives us a map of sets $A:\B(\lam+\rho)_{res}\lra Arr$, which is neither injective nor surjective; we denote the image of this map to be $Arr(\lam)$. The failure of $A$ to be injective is due to the fact that the array $A(\bm)$ does not know $m_1$.

 We define the following operations on the set of arrays of a fixed weight $\kappa$.
\begin{Def}\label{Def: raiselower}
We define the \textbf{raising operators} $e_1,e_2:Arr_{\kappa}\lra Arr_{\kappa}\cup\{0\}$ on the set of arrays given by
\begin{equation*}
e_1\left[\begin{array}{ccccc} a&b&c&b&d\\ &x&y&z&\\&&k&&\end{array}\right]=\left[\begin{array}{ccccc} a-1&b&c+1&b&d-1\\ &x&y+1&z&\\&&k+1&&\end{array}\right]
\end{equation*}
 unless $\min\{a,d\}=0$ in which case $e_1(A) = 0$. Similarly, we define 
\begin{equation*}
e_2\left[\begin{array}{ccccc} a&b&c&b&d\\ &x&y&z&\\&&k&&\end{array}\right]=\left[\begin{array}{ccccc} a&b-1&c+3&b-1&d\\ &x+1&y&z+1&\\&&k+1&&\end{array}\right]
\end{equation*}
where similarly if $b=0$, $e_2(A)=0$. 
Dually, we define the \textbf{lowering operators} $f_1,f_2:Arr_{\kappa}\lra Arr_{\kappa}\cup\{0\}$ as follows:
\begin{equation*}
f_1\left[\begin{array}{ccccc} a&b&c&b&d\\ &x&y&z&\\&&k&&\end{array}\right]=\left[\begin{array}{ccccc} a+1&b&c-1&b&d+1\\ &x&y-1&z&\\&&k-1&&\end{array}\right]
\end{equation*}
 unless $\min\{c,y,k\}=0$ in which case $f_1(A) = 0$. Similarly, we define 
\begin{equation*}
f_2\left[\begin{array}{ccccc} a&b&c&b&d\\ &x&y&z&\\&&k&&\end{array}\right]=\left[\begin{array}{ccccc} a&b+1&c-3&b+1&d\\ &x-1&y&z-1&\\&&k-1&&\end{array}\right]
\end{equation*}
unless $\min\{x,z,k\}=0$ or $c<3$, in which case $f_2(A)=0$.\qed
\end{Def}
It is trivial to check that if $A\in Arr_{\kappa}$ and $e_i(A)\neq 0$, then $e_i(A)\in Arr_{\kappa}$ and similarly for the $f_i$. We need to know that these operations correspond to operations on the set $\B(\lam+\rho)_{res}$.

\begin{Prop}\label{Prop: compat} 
Suppose that $\bm\in\B(\lam+\rho)_{res}$. Let $A(\bm)$ be the array associated to $\bm$. Then there exists $\bn=:e_i(\bm)\in \B(\lam+\rho)_{res}$ such that 
\[
A\left(e_i(\bm)\right) = e_i\left(A(\bm)\right).
\] 
Similarly for the lowering operators $f_1,f_2$. In particular, $Arr(\lam)$ is closed under the raising and lowering operators.
\end{Prop}
\begin{proof}
This is a straightforward check. For example, if we consider $i=1$, and set $$e_1(\bm) = (m_1,m_2-1,m_3,m_4+1,m_5,m_6-1).$$ Since we assume $e_1(A(\bm))\neq0$, we know that this is a Lusztig datum. Moreover, it is clear that it lies in $\B(\lam+\rho)_{res}$ and that $A(e_1(\bm))=e_1(A(\bm))$. 
\end{proof}

For any array $A$, $e_i^NA = 0$ for $N$ large enough. Let $\ep_i(A) =\max\{n\geq0: e_i^n(A)\neq 0\}$. Since $e_1\circ e_2=e_2\circ e_1$, we lose no generality in defining the \emph{head} $\hd(A)$ associated to $A$ to be 
\[
\hd(A)=e_1^{\ep_1(A)}e_2^{\ep_2(A)}(A).
\] 
If $A=A(\bm)$ for some $\bm\in \B(\lam+\rho)_{res}$, the we define $\hd(\bm)$ such that $A(\hd(\bm))=\hd(A(\bm)).$

\subsection{Resonance families}
We define an equivalence relation on $Arr_\kappa$: 
\[
A\sim A'\text{    if and only    } \hd(A) =\hd(A').
\]
By Proposition \ref{Prop: compat}, we see that this restricts to an equivalence relation on $Arr(\lam)$, which we pull back to $\B(\lam+\rho)_{res}$ via the map $A$ as follows: 
\[
\bm\sim \bn \text{   if and only if    }k(\bm) = k(\bm)\text{ and } A(\bm)\sim A(\bn)\in Arr(\lam).
\]  
By taking the weight $k(\bm)$ into account, we repair the failure of $A$ to be injective. For $\bm\in\B(\lam+\rho)_{res}$, we denote its equivalence class $RF(\bm)$, called the \textbf{resonance family} of $\bm$. We record the following obvious lemma.
\begin{Lem}
Every resonance family has a unique element $\bm$ of the form
\[
(m_1,0,0,m_4,0,0),\qquad(m_1,m_2,0,m_4,0,0),\qquad\text{or}\qquad(m_1,0,0,m_4,0,m_6).
\]
We refer to this element as the \textbf{head} of the resonance family $RF(\bm)$. We call a resonance family with head of the form $(m_1,0,0,m_4,0,0)$ a \emph{totally resonant family}.
\end{Lem}

We now consider what these resonance families look like. Note that an element $\bm\in \B(\lam+\rho)_{res}$ is an element of $\B(\lam+\rho)$ if and only if its decoration $k$ vanishes. The following lemma gives a simple test of when a family meets $\B(\lam+\rho)$.
\begin{Lem}\label{Lem: m_4=0}
Let $\bm=(m_i)\in \B(\lam+\rho)_{res}$ such that $m_4=0$. Then $\bm\in \B(\lam+\rho)$.
\end{Lem}
\begin{proof}
We need only show that the decoration $k$ vanishes. Since $\bm\in \B(\lam+\rho)_{res}$, we know that $m_4=\lam_2+m_6-m_5+2k=0$. Then we have
\[
m_5=\lam_2+m_6+2k\leq \lam_2+1.
\]
It follows that $k=0$, and the lemma is proved.
\end{proof}
\begin{Prop}\label{Prop: empty intersection}
Suppose that $RF(\bm)$ is the resonance family with head $\bm=(m_i)$. If $RF(\bm)\cap\B(\lam+\rho)=\emptyset$, then 
\[
RF(\bm)=\{f_1^{t_1}f_2^{t_2}\bm\::\: 0\leq t_1 \leq s_6+1,\;0\leq t_2\leq \min\{s_2,s_5\}+1\}.
\]
\end{Prop}
The content of the claim is that if we assume no element of the resonance family has decoration $k=0$, then $m_4$ is large enough so that the only cause of $f_i(\bn) =0$ for $\bn\in RF(\bm)$ is $s_\al=-1$ for the appropriate $\al$. 
\begin{proof}
  It follows from the definition of the equivalence relation defining a resonance family that 
\[
RF(\bm)=\{f_1^{t_1}f_2^{t_2}\bm\::\: f_1^{t_1}f_2^{t_2}(A(\bm))\neq0\}.
\]
Consider $\bn=(n_i)\in RF(\bm)$ with decoration $k>0$ such that $f_1A(\bn)=0$. We need to see that $s'_6=\lam-n_6=-1$. By the definition of $f_1$, $f_1A(\bn)=0$ implies $\min\{s_6+1,k\}=0$ or $n_4=0$. We know $k>0$ by assumption, and that $n_4>0$ by Lemma \ref{Lem: m_4=0}. We find that $s'_6=-1$.

Now suppose that $\bn=(n_i)\in RF(\bm)$ such that $f_2A(\bn)=0$. By the definition, this means that $\min\{s'_5+1,s'_2+1,k\}=0$ or $n_4<3$. We must show that $\min\{s'_2+1,s'_5+1\}=0$. As above, we need only consider the case $0<n_4<3.$ This implies that 
\[
\lam_2+2\leq \lam_2+n_6+2k=n_5+n_4\leq n_5+2\Rightarrow \lam_2\leq n_5.
\]
If $n_4=1$, then this in equality is strict, resulting in $s'_5=-1$, which is what we wanted to show. If $n_4=2$, then we could have $n_5=\lam_2$. The preceding inequality now says
\[
n_6+2\leq n_6+2k=2\Rightarrow n_6=0.
\]
If $n_2>n_6=0$, then we must have $n_2=1$ so that $s_2'=-1$. If $n_2=n_6=0$, this tells us that $\bn'=f_1\bn\in RF(\bm)$ must have the array
\[
A(\bn')=\left[\begin{array}{ccccc} 1&\lam_2&1&\lam_2&1\\ &1&\lam_1&1&\\&&k'&&\end{array}\right],
\]
where $k'=k-1>0$. This implies that
\[
2k'+1=\lam_2+n_6'+2k'-n_5'=n_4'=1,
\]
a contradiction. This completes the proof.
\end{proof}

\begin{Cor}\label{Cor: empty bound}
Suppose that $RF(\bm)$ is the resonance family with head $\bm=(m_i)$. Let $k>0$ be the decoration of $\bm$. If $RF(\bm)\cap\B(\lam+\rho)=\emptyset$, then 
\[
k>\min\{s_2,s_5\}+s_6+2.
\]
\end{Cor}

The terminology of raising and lowering operators (as well as the notation $e_i$ and $f_i$) is motivated by the relationship between these operators on a resonance family and a certain Kashiwara crystal structure, which we now describe.
Suppose that $RF(\bm)$ is a resonant family with head $\bm$. We define a weight map 
$
 \mathrm{wt}:RF(\bm)\lra \zz^2
$
as follows: if $\bn=f_1^{t_1}f_2^{t_2}\bm$, we set $$\mathrm{wt}(\bn) = (s_6+1-2t_1,\min\{s_2,s_5\}+1-2t_2).$$  
Recall that we have defined
$
\ep_i(\bn) = \max\{k\geq 0: e_i^kA(\bn)\neq0\}$ and set $\varphi_i(\bn) =\ep_i(\bn)+\la h_i,\mathrm{wt}(\bn)\ra.
$

Using the terminology from \cite[Section 4.5]{HongKang}, a crystal graph $(\B,e_i,f_i, \mathbf{wt}, \ep_i,\varphi_i)$ is called \emph{upper seminormal} if 
\[
\ep_i(\bn) = \max\{k: e_i^kx\neq0\} \text{   for all  $x\in \B$ and $i\in I$},
\]
and is called \emph{seminormal} if it is upper seminormal and additionally 
$\varphi_i(x)= \max\{k: f^k_i(x)\neq0\}.
$
\begin{Thm}\label{Thm: crystal structure}
The data $(RF(\bm),e_i,f_i, \mathbf{wt}, \ep_i,\varphi_i)$ satisfy the axioms of a upper seminormal Kashiwara crystal graph of Cartan type $A_1\times A_1$.

 If $RF(\bm)\cap\B(\lam+\rho)=\emptyset$, this crystal is seminormal, and is a highest weight crystal of type $A_1\times A_1$ of highest weight $(s_6+1,\min\{s_2,s_5\}+1)$.
\end{Thm}

\begin{proof}
We refer the reader to \cite[Section 4.5]{HongKang} for the axioms of a Kashiwara crystal of Cartan type $A_1\times A_1$, which are immediately seen to be satisfied by the tuple $(RF(\bm),e_i,f_i, \mathbf{wt}, \ep_i,\varphi_i)$. For example, it follows directly from the formulas for $e_i$ and $f_i$ in Definition \ref{Def: raiselower} that
\[
e_i A(\bn) = A(\bn')\text{   if and only if    } f_i A(\bn') = A(\bn)
\]
whenever $A(\bn),A(\bn')\in RF(\bm)$. For a general resonance family, upper seminormality is immediate from the definition of $\ep_i$.

Now assume that $RF(\bm)\cap\B(\lam+\rho)=\emptyset$. Proposition \ref{Prop: empty intersection} tells us that the definition of weight map $\mathrm{wt}$ is sufficient and takes values ranging in the correct weight set for the highest weight $(s_6+1,\min\{s_2,s_5\}+1)$. It remains only to show that the crystal is seminormal: that is, we need 
$$
\varphi_i(\bn) = \max\{k\geq 0: f_i^kA(\bn)\neq0\},
$$
where $h_i$ is the simple coroot associated to the corresponding $A_1$-factor of the root system.  When $i=1$, our definition gives
\[
\varphi_i(\bn) = \ep_i(\bn)+s_6+1-2t_1 = s_6+1-t_1,
\]
where $\bn=f_1^{t_1}f_2^{t_2}(\bm)$. By Proposition \ref{Prop: empty intersection}, this last value equals $\max\{k\geq 0: f_i^kA(\bn)\neq0\},$ as was to be shown. The argument is identical for $i=2$. 
\end{proof}
In the case that $RF(\bm)\cap\B(\lam+\rho)\neq\emptyset$, the crystal $RF(\bm)$ is contained in an $A_1\times A_1$-crystal that is contained in $Arr(\lam)$. These crystals come equipped with an additional decoration, the parameter $k$. When $k<\min\{s_2+1+s_6+1,s_5+1+s_6+1\}$, after a certain number of lowering operations the parameter $k$ will drop to $0$ prior to reaching the end of the root string. Our definition of the lowering operators $f_i$ then forces the crystal to end at this step so that $RF(\bm)$ is \emph{truncated}. See Figure \ref{Figure 3}. We will see that these truncated crystals contribute to the final evaluation of $I_\lam$.

\section{A geometric Tokuyama-type formula for $G_2$}\label{Section: generalized tok} 
Let $\bm$ be a resonant Lusztig datum and let $RF(\bm)$ be the corresponding resonance family. In this section, we consider sums of the form
\[
\sum_{\bn\in RF(\bm)}I_\lam(\bn).
\]
The first main result is that for those resonance families that are complete Kashiwara crystal graphs, the above sum vanishes. We then compute the sum over those families which do contain Lusztig data in $\B(\lam+\rho)$. In particular, we see the terms outside $\B(\lam+\rho)$ do not vanish.
\begin{Thm}\label{Thm: external vanishing}
Let $RF(\bm)$ be the resonance family such that $RF(\bm)\cap\B(\lam+\rho)=\emptyset$. Then
\[
\sum_{\bn\in RF(\bm)}I_\lam(\bn)=0.
\]
\end{Thm}
\begin{Rem}
This result still holds when we consider spherical Whittaker functions for covers of $G_2$. As previously noted, the sums that appear depend on the cover, but Theorem \ref{Thm: crystal structure} is independent of the degree of the cover.
\end{Rem}
\begin{proof}
Suppose that the resonance family is of weight $\kappa=(k_1,k_2)$. Recall from Theorem \ref{Thm: all together} that for each $\bn\in RF(\bm)$
\begin{equation*}\label{eqn: outside evaluation 2}
I_\lam(\bn)=q^{-k(\bn)}(1-q^{-1})\prod_{\al\neq \ga_3,\ga_4}G(s_\al,n_\al)x_1^{k_1}x_2^{k_2},
\end{equation*}
where $k(\bn)$ is the decoration of $\bn$.
By Proposition \ref{Prop: empty intersection}, 
\[
RF(\bm)=\{f_1^{t_1}f_2^{t_2}\bm\::\: 0\leq t_1 \leq s_6+1,\;0\leq t_2\leq \min\{s_2, s_5\}+1\}.
\]

We first assume that $m_2< m_6$ so that  $s_5=\min\{s_2,s_5\}$. For each value of $t_1$, the $f_2$-root string $\{f_1^{t_1}f_2^{t_2}\bm\::\:\;0\leq t_2\leq s_5+1\}$ contributes
\begin{align*}
q^{t_1-k}(1-q^{-1})G(s_1,m_1)G(s_6-t_1,m_6+t_1)&\left[\sum_{t_2=0}^{s_5+1}q^{t_2}G(s_2-t_2,t_1)G(s_5-t_2,t_2)\right].
\end{align*}
Under these assumptions $G(s_2-t_2,t_1)$ is constant, and we are left with the sum
\[
\sum_{t_2=0}^{s_5+1}q^{t_2}G(s_5-t_2,t_2)= \left[1+\sum_{t_1=1}^{s_5}q^{t_1}(1-q^{-1})-q^{s_5}\right]=0.
\]
Thus the entire sum vanishes.

If we assume $m_2> m_6$, we have $s_2=\min\{s_2,s_5\}$. In this case, for each value of $t_2$, we consider the $f_1$-root string $\{f_1^{t_1}f_2^{t_2}\bm\::\:\;0\leq t_1\leq s_6+1\}$. This contributes
\begin{align*}
q^{t_2-k}(1-q^{-1})G(s_5-t_2,t_2)\left[\sum_{t_1=0}^{s_6+1}q^{t_1}G(s_2-t_2,m_2+t_1)G(s_6-t_1,t_1)\right].
\end{align*}
Again, the term $G(s_2-t_2,m_2+t_1)$ is constant, and the inner sum vanishes, showing that the entire sum vanishes.

Finally, consider the \emph{totally resonant case} $\bm= (m_1,0,0,m_4,0,0)$. In this case, we must sum over the entire rank two crystal to show vanishing. To begin, for each value of $t_1$ the $f_2$-root string $\{f_1^{t_1}f_2^{t_2}\bm\::\:\;0\leq t_2\leq s_5+1\}$ contributes
\begin{align*}
q^{t_1-k}(1-q^{-1})G(s_1,m_1)G(s_6-t_1,t_1)&\left[\sum_{t_2=0}^{s_5+1}q^{t_2}G(s_5-t_2,t_1)G(s_5-t_2,t_2)\right].
\end{align*}
If $t_1=0$, then $G(s_5-t_2,t_1)=1-\delta_{0,s_5-t_2}$, where $\delta$ is the Kronecker delta. This inner sum now has the form 
\begin{align*}
q^{t_1-k}G(s_6-t_1,t_1)\sum_{t_2=0}^{s_5}q^{t_2}G(s_5-t_2,t_2) = q^{s_5-k}.
\end{align*}
For $0<t_1<s_6+1$, $G(s_5-t_2,t_1)=1-q^{-1}-\delta_{0,s_5-t_2}$, and the sum gives
\begin{align*}
q^{t_1-k}G(s_6-t_1,t_1)\sum_{t_2=0}^{s_5+1}q^{t_2}G(s_5-t_2,t_1)G(s_5-t_2,t_2) = q^{t_1+s_5-k}(1-q^{-1}).
\end{align*}
Finally, for $t_1=s_6+1$, we have
\begin{align*}
q^{s_6+1-k}G(-1,\lam_1+1)\sum_{t_2=0}^{s_5+1}q^{t_2}G(s_5-t_2,t_1)G(s_5-t_2,t_2) = -q^{s_6+s_5-k}.
\end{align*}
Summing up these terms, we find that
\[
\sum_{\bn\in RF(\bm)}I_\lam(\bn) =0
\]
in all cases.
\end{proof}
We say that a resonance family is \textbf{$\lam$-relevant} if $RF(\bm)\cap \B(\lam+\rho)\neq \emptyset.$ The are only finitely many such families. Set $RF(\bm)^\circ =RF(\bm)-(RF(\bm)\cap\B(\lam+\rho))$.  The above gives the following ``geometric'' \textbf{Tokuyama-type formula for $G_2$}.
\begin{Thm}\label{Thm: resonance family formula}
Let $\lam\in X^\ast(T)^+$ be a dominant weight for the split complex Lie group of type $G_2$ and let $\chi_\lam$ be its character. Then for $\tau\in T(\cc)$,
\begin{align*}\label{eqn: resonance family formula}
\prod_{\al\in \Phi^+}\left(1-q^{-1}\tau^\al\right)\chi_\lam(\tau)=\sum_{\bn\in \B(\lam+\rho)}&{G_1}(\bn)\tau^{\mu(\bn)-w_0(\lam+\rho)}\\
&+\sum_{\shortstack{$RF(\bm)$ \\$\lam$-$\mathrm{relevant}$}}\left(\sum_{\bn\in RF(\bm)^\circ} G_{res}(\bn)\right)\tau^{\mu(\bn)-w_0(\lam+\rho)},
\end{align*}
where $G_1(\bn)$ and $G_{res}(\bn)$ are as in Theorem \ref{Thm: all together}.
\end{Thm}

To end this section, we give another Tokuyama-type formula by considering the sums over $\lam$-relevant resonance families. As we see below, for a given $\lam$-relevant family that is not totally resonant one can make a ``handedness'' choice to derive simple formulas for augmenting the contributions from $\bn\in RF(\bm)\cap \B(\lam+\rho)$. This reduces the ``correction term'' to the totally resonant case, which does not appear to admit a natural augmentation.

We begin by stating the \textbf{right-hand rule} for augmenting contributions. This is done by summing ``from right to left'' along $f_1$-root strings.
\begin{Prop}\label{Prop: right hand rule}
In the case that $n_3=n_5$, $n_2\neq n_6$ with $\min\{n_2,n_6\}>0$, and $s_3=s_4=0$, we obtain the augmented contribution 
\begin{equation*}\label{eqn: right augmented 1}
\tilde{I}^{r}_\lam(\bn) = \prod_{i=1,2,5}G(s_i,n_i)\left[\prod_{j=3,4,6}G(s_j,n_j)+q^{-1}(1-q^{-1})\right]x_1^{k_1}x_2^{k_2}.
\end{equation*}
If  $n_6>n_2=0$, we must sum along the unique $f_2$-root string. This requires $n_3=n_5>0$ and gives the augmented contribution
\begin{equation*}\label{eqn: rl augmented 2}
\tilde{I}^{r,l}_\lam(\bn) = \prod_{i\neq4,5}G(s_i,n_i)\left[\prod_{j=4,5}G(s_j,n_j)+q^{-1}\right]x_1^{k_1}x_2^{k_2}.
\end{equation*}
For other resonant Lusztig data with $n_2\neq n_6$, we set $\tilde{I}^{r}_\lam(\bn) = I_\lam(\bn)$.
\end{Prop}
We now state the \textbf{left-hand rule} for augmenting contributions by summing over $f_2$-root strings.
\begin{Prop}\label{Prop: left hand rule}
In the case that $n_3=n_5>0$, $n_2\neq n_6$, and $s_3=s_4=0$,  the left-hand augmented contribution is
\begin{equation*}\label{eqn: left augmented}
\tilde{I}^{l}_\lam(\bn) = \prod_{i\neq4,5}G(s_i,n_i)\left[\prod_{j=4,5}G(s_j,n_j)+q^{-1}\right]x_1^{k_1}x_2^{k_2}.
\end{equation*}
If $n_3=n_5=0$, $n_2>n_6>0$, and $s_2=-1$, then there are no $f_2$-root strings and we have the augmented contribution 
\begin{equation*}\label{eqn: right augmented 1}
\tilde{I}^{r,l}_\lam(\bn) = \prod_{i=1,2,5}G(s_i,n_i)\left[\prod_{j=3,4,6}G(s_j,n_j)+q^{-1}(1-q^{-1})\right]x_1^{k_1}x_2^{k_2}.
\end{equation*}
For other resonant Lusztig data with $n_2\neq n_6$, we set $\tilde{I}^{l}_\lam(\bn) = I_\lam(\bn)$.
\end{Prop}
The proofs of these Propositions are computational and straightforward, so we omit the details. 
Note that if $n_2<n_6$,  the left-hand rule is uniform, while the right-hand rule is uniform if $n_2>n_6$.\footnote{This distinction is similar to the Class I, Class II, and totally resonant cases arising in the analysis of Fourier coefficeints of metaplectic Eisenstein series of type $B$ in \cite{FZ15}.} This motivates the following choice.
\begin{Def} Let $\bn\in \B(\lam+\rho)$ and assume that $n_3=n_5$ and $s_3=s_4=0$. If $n_2<n_6$, we set $$\tilde{I}_\lam(\bn) = \tilde{I}^l_\lam(\bn);$$ if $n_2>n_6$, we set $$\tilde{I}_\lam(\bn) = \tilde{I}^r_\lam(\bn).$$

When $\bm$ is not resonant, we set $\tilde{I}_\lam(\bn) = I_\lam(\bn)$ and $\tilde{I}_\lam(\bn)=\tilde{G}(\bn)x^{\mu(\bn)-w_0(\lam+\rho)}$, where $\mu(\bn)=k(\bn)+w_0(\lam+\rho)$.
\end{Def}
Recall that for a $\lam$-relevant resonance family $RF(\bm)$, we set $RF(\bm)^\circ =RF(\bm)-(RF(\bm)\cap\B(\lam+\rho))$. Combining Propositions \ref{Prop: right hand rule}, \ref{Prop: left hand rule}, and Theorem \ref{Thm: resonance family formula} gives us the following formula.
\begin{Thm}\label{Thm: tokuyamatype}
Let $\lam\in X^\ast(T)^+$ be a dominant weight for the split complex Lie group of type $G_2$ and let $\chi_\lam$ be its character. Then for $\tau\in T(\cc)$,
\begin{align*}
\prod_{\al\in \Phi^+}\left(1-q^{-1}\tau^\al\right)\chi_\lam(\tau)=\sum_{\bn\in \B(\lam+\rho)}&\tilde{G}(\bn)\tau^{\mu(\bn)-w_0(\lam+\rho)}\\
&+\sum_{\shortstack{$RF(\bm)$ \\$\lam$-$\mathrm{relevant}$\\{totally resonant}}}\left(\sum_{\bn\in RF(\bm)^\circ} G_{res}(\bn)\right)\tau^{\mu(\bn)-w_0(\lam+\rho)},
\end{align*}
\end{Thm}
We end this section by noting that this formula \emph{does not} recover term-by-term the conjectured formula of \cite{friedlander2015crystal}, even for the case of $\lam=0$. Their formula has recently been established in \cite{defranco}, so it follows that their Tokuyama-type formula agrees with the ``geometric''  formula here derived, but the bad-middle phenomenon does not correspond to the correction terms arising from resonant MV polytopes. Nevertheless, both formulas may be interpreted in the form
\[
\left(\sum_{\nu \in \B(\lam+\rho)}\text{appropriate standard contribution}\right)+\text{finitely many correction terms}.
\]
It would be interesting to give a geometric interpretation for the augmented contributions they assign to bad-middle Littelman paths.

\section{The general notion of resonance}\label{Section: resonance}

In this final section, we discuss how the analysis of the previous section, and in particular the notion of resonant Lusztig data, might be generalized. We make no assumption on the degree $n$ of the covering group. Let $M$ by an MV polytope with $\underline{i}$-Lusztig data $\bm$, and assume that Conjecuture \ref{Conjecture} holds for $\i$. Assume we have made a choice of monomials $\{Y_k\}_{k=1}^N$ by associating to each root the associated $\i$-trail. Generalizing the formula (\ref{eqn: general formula}) in the case $\bm>0$, we have the formula
 \begin{equation*}\label{eqn: general formula1}
 I_\lam(\bm)= \displaystyle\prod_{\al}x_{\al}^{m_\al}\int\cdots\int\prod_k(Y_k,\vp)^{r_k}\prod_{i\in\Del}\psi_\lam\left(\sum_{\pi:\Lam_i\rightarrow w_0s_i\Lam_i}d_\pi Y_{1}^{c_\pi(1)}\cdots Y_{N}^{c_\pi(N)}\right)\prod_{k}dY_k,
 \end{equation*}
where the domain of integration and the precise exponents $c_\pi(k)\in \zz$ depend on the circling pattern of $\bm$. Define the map $$\mathrm{circ}: \zz^N_{\geq0}\lra 2^N,$$ such that $(\mathrm{circ}(\bm))_i=1$ if $m_i>0$ and $0$ if $m_i=0$. This map records the circling pattern of $\bm$.
By construction, the monomials $\{Y_k\}_{k=1}^N$ each occur in the additive character. For each subset $S\subset [N]$, we define $f_S(X_\al)\in \calo[Y_\al]$ to be the sum of monomials containing only $Y_k$ for $k\in S$. For $\nu=\mathrm{circ}(\bm) \in 2^N$, let $\mathcal{P}_{\i}(\nu)$ the be the finest partition of $[N]$ such that $f_S\neq 0$ for $S\in \mathcal{P}_{\i}(\nu)$. Then there exists a decomposition 
\[
\sum_{i\in I}\sum_{\pi:\Lam_i\rightarrow w_0s_i\Lam_i}\vp^{\lam_{i^\ast}}d_\pi Y_{1}^{c_\pi(1)}\cdots Y_{N}^{c_\pi(N)}=\sum_{S\in\mathcal{P}_{\i}(\nu) }f_S(Y_\al),
\]
and it is easy to see that this decomposition only depends on $\mathrm{circ}(\bm)$.
This leads to the decomposition $I_\lam(\bm)=\prod_{S\in \mathcal{P}_{\i}(\nu)}I_{\lam,S}(\bm)$, where
 \begin{equation}\label{eqn: resonance formula}
 I_{\lam,S}(\bm)= \displaystyle\prod_{k\in S}x_{\ga_k}^{m_k}\int\cdots\int\prod_{k\in S}(Y_k,\vp)^{r_k}\psi\left(f_S(Y_\al)\right)\prod_{k\in S}dY_k.
\end{equation}
This decomposition depends only on the choice of word $\i$ and $\mathrm{circ}(\bm)$. These integrals encode the notion of resonance.
\begin{Def} 
 We say $\bm$ is \textbf{resonant} with respect to $\i$, $\lam$, and $S$ if there are two indices $k,l\in S$, $k\neq l$,  such that for any $u\in C^{\i}(\bm)$ 
\[
\val(Y_k)=\val(Y_l).
\]
We say that $\bm$ is \textbf{critically resonant} if there is an $S$ such that $\val(Y_k)=\val(Y_l)$ for all $k,l\in S$.
\end{Def}

Suppose that $M$ is the MV polytope associated to the $\i$-Lusztig data $\bm$. A resonance of $\bm$ corresponds to the existence of certain additional symmetries of $M$. In the case of $G_2$ above, the unique critical resonance corresponded to the equality $m_3=m_5$, though in general resonances will correspond to certain affine relations among the Lusztig data that depend of the cocharacter $\lam$.

At present, we do not know how to conjecture a uniform evaluation of $I_\lam(\bm)$ for resonant data. To do so would require detailed knowledge of the algebraic relations between the monomials arising in $\mathfrak{s}_i(u)$, so is at least as complex as computing all appropriate $\i$-trails. Additionally, the precise relations depend on the circling pattern of $\bm$, increasing the complexity by a factor of $2^N$, though as we saw in Lemma \ref{Lem: simplifying} there is a lot of redundancy in these cases.
This question is the subject of future work. We end with the following conjecture, which generalizes Theorem \ref{Thm: all together} as well as \cite[Theorem 8.4]{McN}.
\begin{Conj}
Fix a long word $\i$. For a given dominant coweight $\lam$, if $I_\lam(\bm)\neq0$ while $(\i,\bm)\notin \B(\lam+\rho)$, then $\bm$ is critically resonant.
\end{Conj}

How the additional structure on the set of critically resonant Lusztig data in Theorem \ref{Thm: crystal structure} generalizes is an interesting and challenging problem. It would be interesting to find a purely geometric representation-theoretic method of isolating these structures.

\appendix
\section{Verification of Conjecture \ref{Conjecture} for $A_3$}\label{App: type A}
In this appendix, we verify Conjecture \ref{Conjecture} in the simplest case. Let $G=\SL(4)$, and consider the reduced expression $\i=(3,2,1,2,3,2)$. We enumerate the three simple roots in the standard way. The induced ordering on the positive roots $\{\ga_k\}_{k=1}^6$ is given by:
\[
\al_1<\al_1+\al_2<\al_1+\al_2+\al_3<\al_3<\al_2+\al_3<\al_2.
\]
Given our choice of $\i$, there is a unique $\i$-trail from $\Lam_2\lra w_0s_2\Lam_2$, two from $\Lam_3\lra w_0s_3\Lam_3$, and four $\i$-trails from $\Lam_1\lra w_0s_1\Lam_1$. This may be checked by applying \cite[Proposition 9.2]{BZ01}, as $\Lam_k$ is always minuscule in type $A$. Then by Proposition \ref{Prop: character sum}, we find that if $u=z_{\i}(b_\bullet)$, then $\fs_2(u)=1/b_6=X_6$, $\fs_3(u)=1/b_5+b_6/b_4b_5=X_5+X_4$, and finally 
\[
\fs_1(u) = \frac{1}{b_3}+\frac{b_4}{b_2b_3}+\frac{b_5b_6}{b_1b_2b_3}+\frac{b_6}{b_2b_3}=X_3+X_2+X_1+\frac{X_2X_4}{X_5}.
\]

 Following the construction of $g_i(t_\al,w_\al)$ in Section \ref{Section: boundary}, we have $g_2(t_\al,w_\al) = t_6$, 
\[
g_2(t_\al,w_\al) = t_5+\frac{t_4w_5}{w_6}, \text{  and  }g_1(t_\al,w_\al)=t_3+\frac{t_2w_3}{w_4}+\frac{t_1w_2w_3}{w_5w_6}+\frac{t_2t_4w_3}{w_4w_6}.
\]
On the other hand, if we apply the algorithm (\ref{Algo 2}), we find that $h_2=g_2$, $h_3=g_3$, but
\[
h_1(t_\al,w_\al)=g_1(t_\al,w_\al)+\frac{t_3w_4}{w_6}\left(1-\frac{t_4}{w_4}\right).
\]
As before, we denote the bounding data $s_k=\val(\vp^{\lam_{i^\ast}}X_k)$.
\begin{Prop}
Let $\i=(3,2,1,2,3,2)$. For any $\i$-Lusztig datum $\bm$ and dominant coweight $\lam$, we have the equality
\begin{equation}\label{what we want}
I_\lam(\bm) =  \int_{C^{\i}(\bm)}f(u)\psi\left(\sum_{i\in I}\vp^{\lam_{i^\ast}}g_i(t_\al,w_\al)\right)du.
\end{equation}
That is, Conjecture \ref{Conjecture} holds in this case.
\end{Prop}
Up to passing to dual long words $\i\mapsto \i^\ast$, this is the only long word of type $A_3$ for which the conjecture is not immediate. The analysis below relies in a precise fashion on the various piecewise linear relations between the bounding data.
\begin{proof}
For simplicity, we assume that $n=1$, though the statement holds in general but with more tedious notation. Note that this is obvious if $t_4=w_4$, so we assume that $m_4=0$. We may express $I_\lam(\bm)$ as 
\[
 \prod_{\al\in \Del^+}(q^{-1}x_\al)^{m_\al}I(s_1,m_1)I(s_6,m_6)J_{2,3,4,5},
\]
where $I(a,b)$ is defined in Section \ref{Section: G2 integrals} and 
\begin{equation*}
J_{2,3,4,5}=\displaystyle\int\int\int\int\psi\left(\vp^{\lam_2}\left( t_5+\frac{t_4w_5}{w_6}\right)+\vp^{\lam_3}\left({t_2w_3}+\frac{t_2t_4w_3}{w_6}+t_3\left(1+\frac{1-t_4}{w_6}\right)\right)\right)dt,
\end{equation*}
and the precise domain of integration depends on the circling pattern of $\bm$.
Note that if $m_6>0$, then $1+(1-t_4)/w_6\in \calo^\times$, so a simple change of variables reduces this to 
\[
\displaystyle\int\int\int\int\psi\left(\vp^{\lam_2}\left( t_5+\frac{t_4w_5}{w_6}\right)+\vp^{\lam_3}\left({t_2w_3}+\frac{t_2t_4w_3}{w_6}+t_3\right)\right)dt,
\]
which matches (\ref{what we want}). If $m_6=0$, the same argument works unless $\lam_3-m_3\leq-1$ and  $\val(2-t_4)>0$. However in this case, we have the inner integration
\[
\int\psi\left(\vp^{\lam_3}t_2w_3\left(1+t_4\right)\right)dt_2=0,
\]
which is also true for (\ref{what we want}). This exhausts the cases, proving the proposition.
\end{proof}






\bibliographystyle{alpha}

\bibliography{bibs}

\newcommand{\etalchar}[1]{$^{#1}$}
\begin{thebibliography}{GRVP15}

\bibitem[A{\etalchar{+}}03]{anderson2003polytope}
Jared~E Anderson et~al.
\newblock A polytope calculus for semisimple groups.
\newblock {\em Duke Mathematical Journal}, 116(3):567--588, 2003.

\bibitem[BBF11a]{BBFybe}
Ben Brubaker, Daniel Bump, and Solomon Friedberg.
\newblock Schur polynomials and the {Y}ang-{B}axter equation.
\newblock {\em Comm. Math. Phys.}, 308(2):281--301, 2011.

\bibitem[BBF11b]{BBF}
Ben Brubaker, Daniel Bump, and Solomon Friedberg.
\newblock Weyl group multiple {D}irichlet series, {E}isenstein series and
  crystal bases.
\newblock {\em Ann. of Math. (2)}, 173(2):1081--1120, 2011.

\bibitem[BBF11c]{BBF1}
Ben Brubaker, Daniel Bump, and Solomon Friedberg.
\newblock {\em Weyl group multiple {D}irichlet series: type {A} combinatorial
  theory}, volume 175 of {\em Annals of Mathematics Studies}.
\newblock Princeton University Press, Princeton, NJ, 2011.

\bibitem[BF15]{BF}
Benjamin Brubaker and Solomon Friedberg.
\newblock Whittaker coefficients of metaplectic {E}isenstein series.
\newblock {\em Geom. Funct. Anal.}, 25(4):1180--1239, 2015.

\bibitem[BFG06]{BFG}
Alexander Braverman, Michael Finkelberg, and Dennis Gaitsgory.
\newblock Uhlenbeck spaces via affine {L}ie algebras.
\newblock In {\em The unity of mathematics}, volume 244 of {\em Progr. Math.},
  pages 17--135. Birkh\"auser Boston, Boston, MA, 2006.

\bibitem[BFZ05]{BFZ}
Arkady Berenstein, Sergey Fomin, and Andrei Zelevinsky.
\newblock Cluster algebras. {III}. {U}pper bounds and double {B}ruhat cells.
\newblock {\em Duke Math. J.}, 126(1):1--52, 2005.

\bibitem[BG01]{BravGaits}
Alexander Braverman and Dennis Gaitsgory.
\newblock Crystals via the affine {G}rassmannian.
\newblock {\em Duke Math. J.}, 107(3):561--575, 2001.

\bibitem[BZ97]{BZ97}
Arkady Berenstein and Andrei Zelevinsky.
\newblock Total positivity in {S}chubert varieties.
\newblock {\em Comment. Math. Helv.}, 72(1):128--166, 1997.

\bibitem[BZ01]{BZ01}
Arkady Berenstein and Andrei Zelevinsky.
\newblock Tensor product multiplicities, canonical bases and totally positive
  varieties.
\newblock {\em Invent. Math.}, 143(1):77--128, 2001.

\bibitem[CG10]{ChintaGunnells}
Gautam Chinta and Paul~E. Gunnells.
\newblock Constructing {W}eyl group multiple {D}irichlet series.
\newblock {\em J. Amer. Math. Soc.}, 23(1):189--215, 2010.

\bibitem[Chh13]{chhaibi2013littelmann}
Reda Chhaibi.
\newblock Littelmann path model for geometric crystals, {W}hittaker functions
  on {L}ie groups and {B}rownian motion.
\newblock {\em arXiv preprint arXiv:1302.0902}, 2013.

\bibitem[CS80]{CassShalika}
W.~Casselman and J.~Shalika.
\newblock The unramified principal series of {$p$}-adic groups. {II}. {T}he
  {W}hittaker function.
\newblock {\em Compositio Math.}, 41(2):207--231, 1980.

\bibitem[DeF18]{defranco}
Mario DeFranco.
\newblock On a conjecture about an analogue of {T}okuyama's theorem for $
  {G}_2$.
\newblock {\em arXiv preprint arXiv:1806.09515}, 2018.

\bibitem[FGG15]{friedlander2015crystal}
Holley Friedlander, Louis Gaudet, and Paul~E Gunnells.
\newblock Crystal graphs, {T}okuyama's theorem, and the
  {G}indikin--{K}arpelevi{\v{c}} formula for $ {G}_2$.
\newblock {\em Journal of Algebraic Combinatorics}, 41(4):1089--1102, 2015.

\bibitem[FZ15]{FZ15}
Solomon Friedberg and Lei Zhang.
\newblock Eisenstein series on covers of odd orthogonal groups.
\newblock {\em Amer. J. Math.}, 137(4):953--1011, 2015.

\bibitem[FZ16]{FZ16}
Solomon Friedberg and Lei Zhang.
\newblock Tokuyama-type formulas for characters of type {B}.
\newblock {\em Israel J. Math.}, 216(2):617--655, 2016.

\bibitem[Gao18]{Gaothesis}
Fan Gao.
\newblock The {L}anglands-{S}hahidi {L}-functions for {B}rylinski-{D}eligne
  extensions.
\newblock {\em Amer. J. Math.}, 140:83--137, 2018.

\bibitem[Gra17]{gray2017metaplectic}
Nathan Gray.
\newblock Metaplectic ice for {C}artan type {C}.
\newblock {\em arXiv preprint arXiv:1709.04971}, 2017.

\bibitem[GRVP15]{GuptaRoyvanPaski}
Vineet Gupta, Uma Roy, and Roger Van~Peski.
\newblock A generalization of {T}okuyama's formula to the {H}all-{L}ittlewood
  polynomials.
\newblock {\em Electron. J. Combin.}, 22(2):Paper 2.11, 18, 2015.

\bibitem[GS15]{GoncharovShen}
Alexander Goncharov and Linhui Shen.
\newblock Geometry of canonical bases and mirror symmetry.
\newblock {\em Invent. Math.}, 202(2):487--633, 2015.

\bibitem[HK02a]{HamelKing}
A.~M. Hamel and R.~C. King.
\newblock Symplectic shifted tableaux and deformations of {W}eyl's denominator
  formula for {${\rm sp}(2n)$}.
\newblock {\em J. Algebraic Combin.}, 16(3):269--300 (2003), 2002.

\bibitem[HK02b]{HongKang}
Jin Hong and Seok-Jin Kang.
\newblock {\em Introduction to quantum groups and crystal bases}, volume~42 of
  {\em Graduate Studies in Mathematics}.
\newblock American Mathematical Society, Providence, RI, 2002.

\bibitem[Kam10]{Kam1}
Joel Kamnitzer.
\newblock Mirkovi\'c-{V}ilonen cycles and polytopes.
\newblock {\em Ann. of Math. (2)}, 171(1):245--294, 2010.

\bibitem[Lam13]{Lam}
Thomas Lam.
\newblock Whittaker functions, geometric crystals, and quantum schubert
  calculus.
\newblock {\em arXiv preprint arXiv:1308.5451}, 2013.

\bibitem[Les17]{leslie2017generalized}
Spencer Leslie.
\newblock A generalized theta lifting, {CAP} representations, and {A}rthur
  parameters.
\newblock {\em arXiv preprint arXiv:1703.02597}, 2017.

\bibitem[Lus94]{Lusztigpositivity}
G.~Lusztig.
\newblock Total positivity in reductive groups.
\newblock In {\em Lie theory and geometry}, volume 123 of {\em Progr. Math.},
  pages 531--568. Birkh\"auser Boston, Boston, MA, 1994.

\bibitem[Lus10]{Luzbook}
George Lusztig.
\newblock {\em Introduction to quantum groups}.
\newblock Modern Birkh\"auser Classics. Birkh\"auser/Springer, New York, 2010.
\newblock Reprint of the 1994 edition.

\bibitem[McN11]{McN}
Peter~J. McNamara.
\newblock Metaplectic {W}hittaker functions and crystal bases.
\newblock {\em Duke Math. J.}, 156(1):1--31, 2011.

\bibitem[McN16]{McNCS}
Peter~J. McNamara.
\newblock The metaplectic {C}asselman-{S}halika formula.
\newblock {\em Trans. Amer. Math. Soc.}, 368(4):2913--2937, 2016.

\bibitem[MV07]{MV07}
I.~Mirkovi\'c and K.~Vilonen.
\newblock Geometric {L}anglands duality and representations of algebraic groups
  over commutative rings.
\newblock {\em Ann. of Math. (2)}, 166(1):95--143, 2007.

\bibitem[Spr09]{Springerbook}
T.~A. Springer.
\newblock {\em Linear algebraic groups}.
\newblock Modern Birkh\"auser Classics. Birkh\"auser Boston, Inc., Boston, MA,
  second edition, 2009.

\bibitem[Tok88]{Tok}
Takeshi Tokuyama.
\newblock A generating function of strict {G}elfand patterns and some formulas
  on characters of general linear groups.
\newblock {\em J. Math. Soc. Japan}, 40(4):671--685, 1988.

\end{thebibliography}

\end{document}